\documentclass[a4paper]{amsart}

\usepackage{amssymb}

\usepackage{mathrsfs}

\usepackage{tikz}

\usepackage{tikz-cd}

\usepackage{graphicx}

\usepackage{hyperref}

\usepackage{latexsym}

\usepackage[all]{xy}

\usepackage{yhmath}

\usepackage{amsmath}

\usepackage{amsthm}

\usepackage{hhline}

\usepackage{multirow}

\numberwithin{equation}{section}

\newtheorem*{thm}{Theorem}
\newtheorem*{prop}{Proposition}
\newtheorem*{lem}{Lemma}
\newtheorem*{cor}{Corollary}
\newtheorem*{conj}{Conjecture}

\theoremstyle{definition}
\newtheorem*{definition}{Definition}
\newtheorem*{example}{Example}

\newtheorem*{claim}{Claim}

\theoremstyle{remark}

\newtheorem*{remark}{Remark}

\newcommand{\N}{\mathbb{N}}
\newcommand{\Z}{\mathbb{Z}}

\newcommand{\g}{\mathfrak{g}}
\newcommand{\bo}{\mathfrak{b}}
\newcommand{\h}{\mathfrak{h}}
\newcommand{\n}{\mathfrak{n}}

\newcommand{\A}{\mathcal{A}_q}

\newcommand{\C}{\mathcal{C}}
\newcommand{\D}{\mathcal{D}}

\newcommand{\Hh}{\mathcal{H}}
\newcommand{\M}{\mathcal{M}}
\newcommand{\K}{\mathcal{K}}
\newcommand{\Nn}{\mathcal{N}}
\newcommand{\Ss}{\mathcal{S}}

\newcommand{\U}{\mathcal{U}}

\newcommand{\Of}{\mathcal{O}}
\newcommand{\Oq}{\mathcal{O}_q}
\newcommand{\Oqhat}{\widehat{\Oq}}

\newcommand{\Oqcap}{\wideparen{\Oq}}

\newcommand{\Uqnhat}{\widehat{U_{q,n}}}

\newcommand{\Uqcap}{\wideparen{U_q}}

\newcommand{\norm}[1]{\left|\left| #1 \right|\right|}
\newcommand{\abs}[1]{\left| #1 \right|}
\newcommand{\ten}[2]{#1\widehat{\otimes}_L#2}

\renewcommand{\labelenumi}{(\roman{enumi})}

\DeclareMathOperator{\Hom}{\text{Hom}}
\DeclareMathOperator{\End}{\text{End}}
\DeclareMathOperator{\gr}{\text{gr}}
\DeclareMathOperator{\id}{\text{id}}

\DeclareMathOperator{\im}{\text{Im}}

\DeclareMathOperator{\htt}{\text{ht}}

\begin{document}

\title{Rigid analytic quantum groups and quantum Arens-Michael envelopes}
\date{}
\author{Nicolas Dupr\'e}
\address{Department of Pure Mathematics and Mathematical Statistics\\
Centre for Mathematical Sciences, Wilberforce Road\\
Cambridge, CB3 0WB, United Kingdom}
\email{nd332@cam.ac.uk}

\begin{abstract}
We introduce a rigid analytification of the quantized coordinate algebra of a semisimple algebraic group and a quantized Arens-Michael envelope of the enveloping algebra of the corresponding Lie algebra, working over a non-archimedean field and when $q$ is not a root of unity. We show that these analytic quantum groups are topological Hopf algebras and Fr\'echet-Stein algebras. We then introduce an analogue of the BGG category $\mathcal{O}$ for the quantum Arens-Michael envelope and show that it is equivalent to the category $\mathcal{O}$ for the corresponding quantized enveloping algebra.
\end{abstract}
\maketitle
\thispagestyle{empty}

\setcounter{tocdepth}{1}
\tableofcontents

\section{Introduction}

\subsection{Background and main results} The study of quantum groups in a $p$-adic analytic setting was first proposed by Soibelman in \cite{Soibelman08}, where he introduced quantum deformations of the algebras of locally analytic functions on $p$-adic Lie groups and of the corresponding distribution algebra of Schneider and Teitelbaum \cite{SchTeit03}. Soibelman conjectured among other things that his quantum distribution algebras are topological Hopf algebras and Fr\'echet-Stein algebras. These latter types of algebras were introduced in \cite{SchTeit03} and play an important role in the theory of locally analytic representations of $p$-adic groups. To the best of our knowledge, Soibelman's conjectures have remained unproved and, since then, apart from the short note \cite{Lyubinin1} and the thesis \cite{wald}, there had been no new constructions or results related to the study of $p$-adic analytic quantum groups until very recently.

We attempt to correct that in this paper by constructing a quantum analogue $\wideparen{U_q(\g)}$, or $\Uqcap$ for short, of the $p$-adic Arens-Michael envelope $\wideparen{U(\g)}$ of the enveloping algebra of the a $p$-adic Lie algebra. Classically, $\wideparen{U(\g)}$ can be identified as the subalgebra of the distribution algebra consisting of distributions supported at the identity, and it is known to be a Fr\'echet-Stein algebra. Its representation theory was first studied in \cite{Schmidt1, Schmidt2} and can be thought of as a first approximation to the locally analytic representation theory of the corresponding $p$-adic Lie group. Our construction of $\Uqcap$ is inspired by the theory developed by Ardakov and Wadsley in \cite{Wadsley2}. In particular we adapt their methods to show that $\Uqcap$ is a Fr\'echet-Stein algebra, see Theorem \ref{defFreSt} and Theorem \ref{FreStProp}. We also show that it is a topological Hopf algebra, see section \ref{Uqhopf}. The algebra $\wideparen{U(\g)}$ is initially defined to be the completion of $U(\g)$ with respect to all submultiplicative seminorms that extend the norm on the ground field $L$. Our algebra $\wideparen{U_q(\g)}$ is defined differently, but we show that it also satisfies a similar universal property: it is the completion of the quantized enveloping algebra $U_q(\g)$ with respect to the submultiplicative seminorms which extend a particular norm on $U_q^0=L[K_\lambda]$, see Corollary \ref{Arens-Michael}.

We also construct a quantum analogue $\Oqcap$ of the algebra of rigid analytic functions on the analytification of a semisimple algebraic group $G$. Specifically, we use the GAGA construction on the quantized coordinate algebra $\Oq:=\Oq(G)$ to obtain an algebra $\Oqcap$ which we show to be a topological Hopf algebra, see section \ref{OqHopf}. We also use techniques based on \cite{Emerton} to prove that $\Oqcap$ is Fr\'echet-Stein, see Proposition \ref{Emerton} and Theorem \ref{OqFS}. Moreover we show that $\Oqcap$ is the completion of $\Oq$ with respect to all submultiplicative seminorms that extend the norm on $L$, see Corollary \ref{Arens-Michael}. Throughout this paper we only work in the case where $q$ is not a root of unity.

We conclude this work by using the Fr\'echet-Stein structure on $\Uqcap$ to construct an analogue of the BGG category $\Of$ for it. Indeed, a particularly important property of Fr\'echet-Stein alegbras is that there is a well behaved abelian category of so-called coadmissible modules over them, which in the geometric setting correspond to global sections of coherent modules over Stein spaces (see \cite[Section 3]{SchTeit03}). There is also a category $\Of$ for quantum groups, see \cite{qcatO}, which is a quantum analogue of the sum of the integral blocks inside the classical BGG category. Finally, there already exists an analogue of category $\Of$ for Arens-Michael envelopes, see \cite{Schmidt2}, and its definition generalises straightforwardly to our quantum setting. Roughly, the category consists of those coadmissible modules over $\Uqcap$ whose weight spaces are finite dimensional and such that the weights are contained in finitely many cosets in the weight lattice. We also require for these modules to be topologically semisimple, a notion which was inspired by work of F\'eaux de Lacroix \cite{Lacroix}. We denote this new category by $\hat{\Of}$. Then we prove that the functor $M\mapsto \Uqcap\otimes_{U_q}M$ is an equivalence of categories between the category $\Of$ for $U_q$ and the category $\hat{\Of}$ (see Theorem \ref{mainthm}). The non-quantum version of this result is the main result of \cite{Schmidt2}, and our proof follows theirs almost identically.

We note that there has been a succesful attempt at constructing a quantum Arens-Michael envelope for $\mathfrak{sl_2}$ and proving that it is a Fr\'echet-Stein algebra in \cite{Lyubinin1}, but the general case hasn't been tackled before. Although the object we construct is the same as theirs for $\mathfrak{sl_2}$, our constructions and proofs are different. Very recently, Smith \cite{Craig2} has constructed certain analytic quantum groups using Nichols algebras. It would be interesting to compare our algebras to his.

\subsection{Future research} We ultimately aim to develop a theory of $D$-modules to understand representations of $\wideparen{U_q(\g)}$. In the classical setting, the Arens-Michael envelope $\wideparen{U(\g)}$ can be viewed as a quantization of the algebra of rigid analytic functions on $\g^*$, and is the right object to consider in order to obtain a Beilinson-Bernstein type equivalence, see \cite{Wadsley2, Wadsley3, Ardakov1}. We are working on a Beilinson-Bernstein type equivalence in our context, and this motivates our choice of working with $\wideparen{U_q(\g)}$. Indeed there exists a theory \cite{QDmod1, QDmod2} of quantum $D$-modules and a Beilinson-Bernstein theorem for representations of $U_q(\g)$ developed by Backelin and Kremnizer, and there is also an analogous quantum Beilinson-Bernstein theorem due to Tanisaki \cite{Tanisaki}. In \cite{Nico1} we will begin to adapt the Backelin-Kremnizer theory of quantum $D$-modules to our setting.

\subsection{Structure of the paper} In section 2 we recall the basic facts and definitions about quantum groups that we will need. In section 3 we define the algebras $\Uqcap$ and $\Oqcap$ and use standard results from functional analysis to prove that they are Fr\'echet Hopf algebras. In section 4, we develop general criteria to establish that certain algebras are Fr\'echet-Stein. Specifically, we use the notion of a deformable algebra from \cite{Wadsley1} and adapt two useful criteria for flatness from \cite{Wadsley2, Emerton} to our setting. We then use those to prove that $\Uqcap$ and $\Oqcap$ are Fr\'echet-Stein algebras. In doing so, we prove a PBW type theorem for certain lattices inside $U_q$ and obtain universal poperties for $\Uqcap$ and $\Oqcap$. Finally, in Section 5, we introduce the notion of a topologically semisimple module. We then use this to define the category $\hat{\Of}$ and investigate its properties. In particular, we construct Verma modules and highest weight modules for $\Uqcap$. We then show that this category is equivalent to the category $\Of$ for $U_q$. One of the main ingredients is a form of block decomposition by central characters.

\subsection{Acknowledgements} The results of this paper form part of the author's PhD thesis, which is being produced under the supervision of Simon Wadsley. We would like to thank him for his help, encouragement and patience. We would also like to thank Andreas Bode for his continued interest in our work and for many useful conversations. We thank Tobias Schmidt for pointing out a reference.

\subsection{Conventions and notation} Throughout $L$ will denote a complete discrete valuation field of characteristic 0 with valuation ring $R$, uniformizer $\pi$ and residue field $k$. We fix a unit element $q\in R$ which is \emph{not} a root of unity, and we assume that $q\equiv 1\pmod{\pi}$. Unless explicitly stated otherwise, the term ``module'' will be used to mean \emph{left} module, and Noetherian rings are both left and right Noetherian. Given a ring homomorphism $A\to B$, we will say that $B$ is flat over $A$ to mean that it's both left flat and right flat.

All of our filtrations on modules or algebras will be positive and exhaustive unless specified otherwise. Furthermore, given a ring $S$, a subring $F_0S$ such that $S$ is generated over $F_0S$ by some elements $x_1, \ldots, x_n$ which normalise $F_0S$, and integers $d_1, \ldots, d_n\geq 1$, there is a ring filtration on $S$ by $F_0S$-submodules given by setting
$$
F_tS=F_0S\cdot\{x_{i_1}\cdots x_{i_r}: \sum_{j=1}^r d_{i_j}\leq t\}
$$
for each $t\geq 0$. In such a setting, we will simply say `the filtration given by assigning each $x_i$ degree $d_i$' to refer to this filtration.

Following \cite[Def 2.7]{Wadsley1}, an $R$-submodule $W$ of an $L$-vector space $V$ will be called a \emph{lattice} if the map $W\otimes_R L\to V$ is an isomorphism and $W$ is $\pi$-adically separated, i.e $\bigcap_{n\geq 0}\pi^n W=0$. Also, for any $R$-module $M$, we denote by $\widehat{M}:=\varprojlim M/\pi^n M$ its $\pi$-adic completion.

Finally, we let $\g$ be a complex semisimple Lie algebra. We fix a Cartan subalgebra $\h\subseteq \g$ contained in a Borel subalgebra. We choose a positive root system and we denote the simple roots by $\alpha_1, \ldots, \alpha_n$. Let $C=(a_{ij})$ denote the Cartan matrix. We let $G$ be the simply connected semisimple algebraic group corresponding to $\g$, and we let $B$ be the Borel subgroup corresponding to the positive root system, and let $N\subset B$ be its unipotent radical. Let $\bo=\text{Lie}(B)$ and $\n=\text{Lie}(N)$. Let $W$ be the Weyl group of $\g$, and let $\langle\, , \rangle$ denote the standard normalised $W$-invariant bilinear form on $\h^*$. Let $P\subset\h^*$ be the weight lattice and $Q\subset P$ be the root lattice. Let $d$ be the smallest natural number such that $\langle  \mu, P\rangle\subset \frac{1}{d}\Z$ for all $\mu\in P$. Let $d_i=\frac{\langle \alpha_i,\alpha_i\rangle}{2}\in\{1, 2, 3\}$ and write $q_i:=q^{d_i}$.

We make the following two assumptions. First, we assume that $q^\frac{1}{d}$ exists in $R$ and that $q^\frac{1}{d}\equiv 1\pmod{\pi}$. Secondly, we assume that $p>2$ and, if $\g$ has a component of type $G_2$, we furthermore restrict to $p>3$.

All the above algebraic groups and Lie algebras have $k$-forms, and we write $G_k, \g_k,\ldots$ etc to denote them.

\section{Preliminaries}

\subsection{Quantized enveloping algebra}\label{PrelimonUq} We begin by reviewing basic facts about quantized enveloping algebras (see eg \cite[Chapter I.6]{BroGoo02} for more details). We recall some usual notation for quantum binomial coefficients. For $n\in\Z$ and $t\in L$, we write $[n]_t:=\frac{t^n-t^{-n}}{t-t^{-1}}$. We then set the quantum factorial numbers to be given by $[0]_t!=1$ and $[n]_t!:=[n]_t[n-1]_t\cdots [1]_t$ for $n\geq 1$. Then we define
$$
{n \brack i}_t:=\frac{[n]_t!}{[i]_t![n-i]_t!}
$$
when $n\geq i\geq 1$.

\begin{definition}
The simply connected quantized enveloping algebra $U_q(\g)$ is defined to be the $L$-algebra with generators $E_{\alpha_1},\ldots, E_{\alpha_n}$, $F_{\alpha_1}, \ldots, F_{\alpha_n}$, $K_\lambda$, $\lambda\in P$, satisfying the following relations:
\begin{align*}
& K_\lambda K_\mu=K_{\lambda+\mu},\quad K_0=1,\\
& K_\lambda E_{\alpha_i} K_{-\lambda}=q^{\langle\lambda, \alpha_i\rangle}E_{\alpha_i},\quad K_\lambda F_{\alpha_i} K_{-\lambda}=q^{-\langle\lambda, \alpha_i\rangle}F_{\alpha_i},\\
& [E_{\alpha_i}, F_{\alpha_j}]=\delta_{ij}\frac{K_{\alpha_i}-K_{-\alpha_i}}{q_i-q_i^{-1}},\\
& \sum_{l=0}^{1-a_{ij}}(-1)^l {1-a_{ij}\brack l}_{q_i} E_{\alpha_i}^{1-a_{ij}-l}E_{\alpha_j}E_{\alpha_i}^l=0\quad (i\neq j),\\
& \sum_{l=0}^{1-a_{ij}}(-1)^l {1-a_{ij}\brack l}_{q_i} F_{\alpha_i}^{1-a_{ij}-l}F_{\alpha_j}F_{\alpha_i}^l=0\quad (i\neq j).
\end{align*}
\end{definition}

We will also abbreviate $U_q(\g)$ to $U_q$ when no confusion can arise as to the choice of Lie algebra $\g$. We can define Borel and nilpotent subalgebras, namely $U_q^{\geq 0}$ is the subalgebra generated by all the $K's$ and the $E's$, and $U_q^+$ is the subalgebra generated by all the $E's$. Similarly we can define $U_q^{\leq 0}$ as the algebra generated by all the $K$'s and the $F$'s, and $U_q^-$ is the subalgebra generated by the $F$'s. There is also a Cartan subalgebra given by $U_q^0:=L[K_\lambda: \lambda\in P]$, which is isomorphic to the group algebra $LP$. There is an algebra automorphism $\omega$ of $U_q$ defined by $\omega(E_{\alpha_i})=F_{\alpha_i}$, $\omega(F_{\alpha_i})=E_{\alpha_i}$ and $\omega(K_\lambda)=K_{-\lambda}$.

Recall that $U_q$ is a Hopf algebra with operations given by
$$
\begin{array}{l l l}
\Delta(K_\lambda)=K_\lambda\otimes K_\lambda\quad & \varepsilon(K_\lambda)=1\quad & S(K_\lambda)=K_{-\lambda}\\
\Delta(E_{\alpha_i})=E_{\alpha_i}\otimes 1 + K_{\alpha_i}\otimes E_{\alpha_i}\quad & \varepsilon(E_{\alpha_i})=0\quad & S(E_{\alpha_i})=-K_{-\alpha_i}E_{\alpha_i}\\
\Delta(F_{\alpha_i})=F_{\alpha_i}\otimes K_{-\alpha_i} + 1\otimes F_{\alpha_i}\quad & \varepsilon(F_{\alpha_i})=0\quad & S(F_{\alpha_i})=-F_{\alpha_i}K_{\alpha_i}
\end{array}
$$
for $i=1,\ldots, n$ and all $\lambda\in P$. Then $U_q^{\geq 0}$ and $U_q^{\leq 0}$ are sub-Hopf algebras of $U_q$.

We now recall the construction that leads to the PBW basis for $U_q$ (see \cite[Chapter 8]{Jantzen} for more details). Firstly, we have a triangular decomposition
$$
U_q\cong U_q^-\otimes_L U_q^0\otimes_L U_q^+
$$
so that it is sufficient to find bases for $U_q^\pm$. In order to obtain a basis for $U_q^+$, we consider the action of the braid group on $U_q$ due to Lusztig. Firstly, we recall the usual notation
$$
E_{\alpha_i}^{(s)}:=\frac{E_{\alpha_i}^s}{[s]_{q_i}!},\quad F_{\alpha_i}^{(s)}:=\frac{F_{\alpha_i}^s}{[s]_{q_i}!},
$$
for any integer $s\geq 0$. The braid group action as algebra automorphisms of $U_q$ is then defined by
\begin{align*}
& T_iE_{\alpha_i}=-F_{\alpha_i}K_{\alpha_i}\\
& T_iF_{\alpha_i}=-K_{-\alpha_i}E_{\alpha_i}\\
& T_iE_{\alpha_j}=\sum_{s=0}^{-a_{ij}}(-1)^{s-a_{ij}}q_i^{-s}E_i^{(-a_{ij}-s)}E_jE_i^{(s)}\quad (i\neq j)\\
& T_iF_{\alpha_j}=\sum_{s=0}^{-a_{ij}}(-1)^{s-a_{ij}}q_i^sF_i^{(s)}F_jF_i^{(-a_{ij}-s)}\quad (i\neq j)\\
& T_i K_\lambda=K_{s_i(\lambda)}
\end{align*}
The above action can be extended to construct operators $T_w$ for any element $w\in W$. Indeed, if $w=s_{i_1}\cdots s_{i_s}$ is a reduced expression for $w$, then let $T_w=T_{i_1}T_{i_2}\cdots T_{i_s}$. Moreover, if $w=w_1w_2$ where $\ell(w)=\ell(w_1)+\ell(w_2)$ then $T_w=T_{w_1}T_{w_2}$.

Let $N$ denote the number of positive roots of $\g$. Let $w_0\in W$ be the unique element of longest length and choose a reduced expression $w_0=s_{i_1}\cdots s_{i_N}$. Recall that then
$$
\beta_1:=\alpha_{i_1}, \beta_2:=s_{i_1}(\alpha_{i_2}), \ldots, \beta_N:=s_{i_1}\cdots s_{i_{N-1}}(\alpha_{i_N})
$$
are all the positive roots of $\g$ in some order. Then we define elements $E_{\beta_1}, \ldots, E_{\beta_N}$ of $U_q$ by
$$
E_{\beta_j}:=T_{i_1}\cdots T_{i_{j-1}}(E_{\alpha_{i_j}}).
$$
If in particular $\beta_j=\alpha_t$ is a simple root, then we have $E_{\beta_j}=E_{\alpha_t}$. Note that we have in general $K_\lambda E_{\beta_j}K_{-\lambda}=q^{\langle\lambda, \beta_j\rangle}E_{\beta_j}$.

Then the set of all ordered monomials $E_{\beta_1}^{m_1}\cdots E_{\beta_N}^{m_N}$ forms a basis for $U_q^+$. This depends on a choice of reduced expression for $w_0$ so we fix one for the rest of this paper. We now let $F_{\beta_j}:=\omega(E_{\beta_j})$ and the corresponding monomials in the $F$'s will form a basis of $U_q^-$. The triangular decomposition immediately gives a PBW type basis for $U_q$, namely the basis consists of all ordered monomials
$$
F_{\beta_1}^{n_1}\cdots F_{\beta_N}^{n_N}K_\lambda E_{\beta_1}^{m_1}\cdots E_{\beta_N}^{m_N}
$$
for $m_i, n_j\in \Z_{\geq 0}$ and $\lambda\in P$. For short we will write
$$
M_{\boldsymbol{r}, \boldsymbol{s}, \lambda}:=\boldsymbol{F^r}K_\lambda\boldsymbol{E^s}
$$
where $\boldsymbol{r}, \boldsymbol{s}\in\Z^N_{\geq 0}$. We recall that the \emph{height} of such a monomial is defined to be
$$
\htt(M_{\boldsymbol{r}, \boldsymbol{s}, \lambda}):=\sum_{j=1}^N(r_j+s_j)\htt(\beta_j)
$$
where $\htt(\beta):=\sum_{i=1}^n a_i$ for a positive root $\beta=\sum_i a_i \alpha_i$. This gives rise to a positive algebra filtration on $U_q$  defined by
$$
F_iU_q:=L\text{-span}\{M_{\boldsymbol{r}, \boldsymbol{s}, \lambda} : \htt(M_{\boldsymbol{r}, \boldsymbol{s}, \lambda})\leq i\}.
$$
From now on we will always refer to this filtration as the \emph{height filtration} on $U_q$. It can be extended to a multifiltration as follows (see \cite[Section 10]{DeConProc} for details): the associated graded algebra $U^{(1)}=\gr U_q$ with respect to the above filtration can be seen to have the same presentation as $U_q$, with the exception that now all the $E$'s commute with all the $F$'s. Moreover it has the same vector space basis, by which we mean the basis for $U^{(1)}$ is consists of the symbols of the basis elements for $U_q$. If we impose the reverse lexicographic orderin ordering on $\Z^{2N}_{\geq 0}$, then we can filter $U^{(1)}$ by assigning to each monomial $M_{\boldsymbol{r}, \boldsymbol{s}, \lambda}$ the degree $(r_1,\ldots,r_N, s_1,\ldots, s_N)$. In other words for each $\boldsymbol{d}\in\Z^{2N}_{\geq 0}$, we set $F_{\boldsymbol{d}}U^{(1)}$ to be the span of the monomials $M_{\boldsymbol{r}, \boldsymbol{s}, \lambda}$ such that $(r_1,\ldots,r_N, s_1,\ldots, s_N)\leq \boldsymbol{d}$. This is an algebra multi-filtration, and we denote the corresponding associated graded algebra of $U^{(1)}$ by $U^{(2N+1)}$.

\begin{thm}\emph{(\cite[Proposition 10.1]{DeConProc})}
The algebra $U^{(2N+1)}$ is the $L$-algebra with generators
$$
E_{\beta_1},\ldots, E_{\beta_N}, F_{\beta_1},\ldots, F_{\beta_N}, K_\lambda (\lambda\in P)
$$
and relations
\begin{align*}
& K_\lambda K_\mu=K_{\lambda+\mu},\quad K_0=1,\\
& K_\lambda E_{\beta_i}=q^{\langle\lambda, \beta_i\rangle}E_{\alpha_i}K_\lambda,\quad K_\lambda F_{\beta_j}=q^{-\langle\lambda, \beta_j\rangle}F_{\beta_j}K_\lambda,\\
& E_{\beta_i}F_{\beta_j}=F_{\beta_j}E_{\beta_i}\\
& E_{\beta_i} E_{\beta_j}=q^{\langle\beta_i, \beta_j\rangle}E_{\beta_j}E_{\beta_i}, \quad F_{\beta_i} F_{\beta_j}=q^{\langle\beta_i, \beta_j\rangle}F_{\beta_j}F_{\beta_i}
\end{align*}
for $\lambda, \mu\in P$ and $1\leq i,j\leq N$.
\end{thm}

\subsection{Quantized coordinate rings} \label{prelimonOq} We now recall the construction of the quantized coordinate algebra $\Oq$. For any module $M$ over an $L$-Hopf algebra $H$, and for any $f\in H^*$ and $v\in M$, the matrix coefficient $c^M_{f,v}\in H^*$ is defined by
$$
c^M_{f,v}(x):=f(xv)\quad\quad\text{for } x\in H.
$$
Also recall from \cite[Theorem 5.10]{Jantzen} that for each $\lambda\in P$ there is a unique irreducible representation of type $\boldsymbol{1}$, $V(\lambda)$, of $U_q$ and that these form a complete list of such representations. The module $V(\lambda)$ has a highest weight vector $v_\lambda$ of weight $\lambda$ and we can pick a weight basis, which we will write as $\{v_i\}$ for short, and we will write $\{f_i\}$ for the corresponding dual basis.

The quantized coordinate ring $\Oq$ is then defined to be the $L$-subalgebra of $U_q^\circ$ generated by all matrix coefficients of the modules $V(\lambda)$ for $\lambda\in P^+$. In other words, it is the algebra generated by the $c^{V(\lambda)}_{f_i,v_j}$ where $\lambda\in P^+$ (this does not depend on our choice of weight basis). Hence $\Oq$ is the algebra of matrix coefficients of finite dimensional type $\boldsymbol{1}$ representations of $U_q$.

Furthermore $\Oq$ is actually generated by the matrix coefficients of the modules $V(\varpi_1), \ldots, V(\varpi_r)$ (see \cite[Proposition I.7.8]{BroGoo02}). It is a sub-Hopf algebra of $U_q^\circ$ (see \cite[Lemma I.7.3]{BroGoo02}) with Hopf algebra maps given by:
\begin{equation}\label{Hopfrelations}
\varepsilon(c^{V(\lambda)}_{f_i,v_j})=f_i(v_j)=\delta_{ij},\quad S(c^{V(\lambda)}_{f_i,v_j})=c^{V(\lambda)^*}_{v_j,f_i}, \quad \Delta(c^{V(\lambda)}_{f_i,v_j})=\sum_k c^{V(\lambda)}_{f_i,v_k}\otimes c^{V(\lambda)}_{f_k,v_j}
\end{equation}
where we have $V(\lambda)^*\cong V(-w_0\lambda)$.

We conclude by describing certain $q$-commutator relations in $\Oq$. For each $i$ we let $B_i$ denote our basis of $V(\varpi_i)$ and $B_i^*$ denote the dual basis. By the above $\Oq$ is generated by the set
$$
X=\{c^{V(\varpi_i)}_{f, v}:i=1, \ldots n, f\in B_i^*, v\in B_i\}.
$$
From \cite[I.8.16-I.8.18]{BroGoo02}, we may order $X$ into a list $x_1, \ldots, x_r$ so that there exists $q_{ij}\in R^\times$, equal to some power of $q$, and $\alpha^{st}_{ij}, \beta^{st}_{ij}\in L^\times$ such that
$$
x_ix_j=q_{ij}x_jx_i+\sum_{s=1}^{j-1}\sum_{t=1}^{r} (\alpha^{st}_{ij} x_s x_t+ \beta^{st}_{ij}x_tx_s)
$$
for $1\leq j<i\leq r$.

One can use these relations to deduce that $\Oq$ is Noetherian. Indeed let $F_{\cdot}$ denote the filtration on $\Oq$ obtained by giving $x_i$ degree $d_i=2^r-2^{r-i}$. That is we set
$$
F_t\Oq=L\text{-span}\{x_{i_1}\cdots x_{i_n}: \sum_{j=1}^{n} d_{i_j}\leq t\}.
$$
These degrees are chosen so that whenever $i> j> s$ and $t\leq r$, we always have $d_i+d_j>d_s+d_t$. Then we have:

\begin{thm}\emph{(\cite[Proposition I.8.17 \& Theorem I.8.18]{BroGoo02})} With respect to the above filtration, $\gr \Oq$ is a $q$-commutative $L$-algebra and so Noetherian.
\end{thm}

Here we used the following (recall we assumed that $q^{\frac{1}{d}}\in R$):

\begin{definition} Let $A$ be an $R$-algebra. We say that a set of elements $x_1, \ldots, x_m\in A$ \emph{$q$-commute} if for all $1\leq i, j\leq m$ we have $x_ix_j=q^{n_{ij}}x_jx_i$ for some $n_{ij}\in\frac{1}{d}\Z$. Suppose that $S$ is an $R$-subalgebra of $A$. We say that $A$ is a \emph{$q$-commutative $S$-algebra} if $A$ is finitely generated over $S$ by elements $x_1, \ldots, x_m$  which normalise $S$ and which $q$-commute.
\end{definition}

From a noncommutative analogue of Hilbert's basis theorem \cite[Theorem 1.2.10]{noncomalg} and by induction, we immediately deduce:

\begin{lem} Let $A$ be a $q$-commutative $S$-algebra as above. If $S$ is Noetherian then so is $A$.
\end{lem}

\subsection{Deformable algebras and modules}\label{deformablealg}

Recall from \cite[Definition 3.5]{Wadsley1} that a positively $\Z$-filtered $R$-algebra $A$ with $F_0A$ an $R$-subalgebra of $A$ is said to be a \emph{deformable $R$-algebra} if $\gr A$ is a flat $R$-module and $A$ is a lattice in $A_L$. Its \emph{$n$-th deformation} is the subring
$$
A_n=\sum_{i\geq 0} \pi^{ni}F_iA.
$$
A morphism between deformable $R$-algebras is a filtered $R$-algebra homomorphism. There is nothing stopping us from making the exact same definitions with filtered $R$-modules, i.e. for a filtered $R$-module $M$ we can define $M_n=\sum_{i\geq 0} \pi^{ni}F_iM$ and say that $M$ is deformable if $\gr M$ is a flat $R$-module and $M$ is a lattice in $M_L$.

\begin{remark}
Note that forcing deformable algebras to be $\pi$-adically separated is not a very big restriction, for instance it always holds when $A$ is a Noetherian domain as long as $\pi$ is not a unit by \cite[Proposition I.4.4.5]{Zariskian}.
\end{remark}

We can then extend known results with identical proofs:

\begin{lem} Let $M$ be a deformable $R$-module. Then
\begin{enumerate}
\item \emph{(\cite[Lemma 3.5]{Wadsley1})} For all $n\geq 0$, $M_n$ is also deformable, with filtration
$$
F_jM_n:=M_n\cap F_jM=\sum_{i=0}^j  \pi^{ni}F_iM,
$$
and there is a natural isomorphism $\gr M_n\cong \gr M$.
\item \emph{(\cite[Lemma 6.4(a)]{Wadsley2})} $M_1\cap\pi^tM=\sum_{i\geq t}\pi^iF_iM$ for any $t\geq0$;
\item \emph{(\cite[Lemma 6.4(b)]{Wadsley2})} $(M_n)_m=M_{n+m}$ for any $n, m\geq0$.
\end{enumerate}
\end{lem}

We also record here a useful fact about tensor products that we will need later on. Recall that given two filtered $R$-modules $M$ and $N$, we can give $M\otimes_R N$ a tensor filtration, where $F_t(M\otimes_R N)$ is generated as an $R$-module by all elementary tensors $m\otimes n$ such that $m\in F_i M$ and $n\in F_jN$ where $i+j=t$.

\subsection{Lemma}\label{Schneiderspropn} If $M$ and $N$ are torsion-free filtered $R$-modules, then $(M\otimes_R N)_n=M_n\otimes_R N_n$ for all $n\geq 0$.

\begin{proof}
Since $M$ and $N$ are flat, we have an injective homomorphism $M_n\otimes_R N_n\to M\otimes_R N$. Identifying $M_n\otimes_R N_n$ with its image, we may assume that $M_n\otimes_R N_n$ and $(M\otimes_R N)_n$ both are submodules of $M\otimes_R N$. But now, for each $t\geq 0$, we have in $M\otimes_R N$ that $\pi^{tn}(a\otimes b)=\pi^{in}a\otimes \pi^{jn}b$, where $a\in F_iM$ and $b\in F_jN$ and $i+j=t$. Thus we see that $(M\otimes_R N)_n=M_n\otimes_R N_n$ since $t$ was arbitrary.
\end{proof}

Hence $M\mapsto M_n$ is a monoidal endofunctor of the category of torsion-free filtered $R$-modules.

\section{Completions of quantum groups}

\subsection{The functor $M\mapsto \wideparen{M_L}$}\label{Mcap} We begin by recalling the constructions from \cite[Section 6.7]{Wadsley2}, which were written in terms of $R$-algebras but extend identically to $R$-modules. If $M$ is a torsion-free filtered $R$-module, let $\widehat{M_{n, L}}:= \widehat{M_n}\otimes_R L$ for each $n\geq 0$. This is an $L$-Banach space, with unit ball $\widehat{M_n}$. To simplify notation, we write $\widehat{M_L}$ for $\widehat{M_{0, L}}$.

Now, we have a descending chain
$$
M=M_0\supset M_1\supset M_2\supset\cdots
$$
which induces an inverse system of $L$-Banach spaces and continuous linear maps
$$
\widehat{M_{L}}=\widehat{M_{0, L}}\leftarrow \widehat{M_{1, L}}\leftarrow\widehat{M_{2, L}}\leftarrow\cdots
$$
whose inverse limit we write as
$$
\wideparen{M_L}:=\varprojlim \widehat{M_{n, L}}.
$$
The maps $\wideparen{M_L}\rightarrow \widehat{M_{n, L}}$ induce continuous seminorms $\vert\vert\cdot\vert\vert_n$ on $\wideparen{M_L}$, such that the completion of $\wideparen{M_L}$ with respect to $\vert\vert\cdot\vert\vert_n$ is $\widehat{M_{n, L}}$. Hence $\wideparen{M_L}$ is an $L$-Fr\'echet space. Thus we have defined a functor $M\mapsto \wideparen{M_L}$ from torsion-free filtered $R$-modules to the category of $L$-Fr\'echet spaces.

We now apply the above construction to certain lattices in the quantum algebras we've defined.  Let $U$ denote the De Concini-Kac integral form of the quantum group, which here means the $R$-subalgebra of $U_q$ generated by the $E_{\alpha_i}$'s, $F_{\alpha_j}$'s and the $K$'s. We filter this algebra by setting $F_0U=R[K_\lambda: \lambda\in P]$ and giving each $E_\alpha$ and $F_\alpha$ degree 1. Then each deformation $U_n$ is the $R$-subalgebra of $U_q$ generated by the $\pi^nE_{\alpha_i}$'s, $\pi^nF_{\alpha_j}$'s and the $K$'s.

Note that by the definition of the Hopf algebra structure on $U_q$, we see that each $U_n$ is an $R$-Hopf subalgebra of $U_q$.

\begin{definition} We let $\Uqnhat:=\widehat{U_{n,L}}$ and $\Uqcap:=\wideparen{U_L}=\varprojlim \Uqnhat$ where we give $U$ the above filtration.
\end{definition}

We now consider a different integral form of $U_q$, namely Lusztig's integral form. It is the $R$-subalgebra $U_R^{\text{res}}$ of $U_q$ generated by $K_\lambda^{\pm 1}$ ($\lambda\in P$) and all $E_{\alpha_i}^{(r)}$ and $F_{\alpha_i}^{(r)}$ for $r\geq 1$ and $1\leq i\leq n$. It is an $R$-Hopf subalgebra of $U_q$. Moreover, by \cite[Theorem 6.7]{Lusztig1} $U_R^{\text{res}}$ has a triangular decomposition and a PBW type basis, so that it is free over $R$. Note that, since $U\subset U_R^{\text{res}}$, it immediately implies that $U$ is $\pi$-adically separated.

We now define $\A$ to be the $R$-subalgebra of $\Hom_R(U_R^{\text{res}}, R)$ generated by the matrix coefficients of all the $R$-finite free integrable $U_R^{\text{res}}$-modules of type $\boldsymbol{1}$ (see \cite[Section 1]{Andersen}). These representations are $R$-lattices inside finite dimensional $U_q$-modules of type $\boldsymbol{1}$ and are closed under taking tensor products and duals, hence $\A$ is an $R$-Hopf algebra and, after extending scalars, we see that the matrix coefficients generating $\A$ are in $\Oq$. This realises $\A$ as an $R$-Hopf subalgebra of $\Oq$. Note that $\Hom_R(U_R^{\text{res}}, R)$ is evidently $\pi$-adically separated hence so is $\A$: if $f\in \bigcap \pi^n \Hom_R(U_R^{\text{res}}, R)$ then $\im(f)\subseteq \bigcap \pi^n R=0$.

By inducing one dimensional representations from Borel subalgebras, we get lattices in all the fundamental representations $V(\varpi_i)$ which are integrable $U_R^{\text{res}}$-modules (see \cite[Section 3.3]{Andersen}). So we see that by choosing weight bases for these lattices, the generators $x_1, \ldots, x_r$ of $\Oq$ from \ref{prelimonOq} lie in $\A$. Moreover by \cite[Proposition \& Remark 12.4]{Andersen}, $\A$ is generated by $x_1, \ldots, x_r$ as an $R$-algebra. We now give the filtration to $\A$ given by assigning to each $x_i$ degree 1. So the $n$-th deformation is the $R$-subalgebra generated by all the $\pi^n x_i$.

\begin{definition} We let $\Oqcap:=\wideparen{(\A)_L}$ where we give $\A$ the above filtration.
\end{definition}

We will now show that $\Uqcap$ and $\Oqcap$ are Hopf algebras in a suitable sense, when working in the category of $L$-Fr\'echet spaces.

\subsection{Completed tensor products}\label{funal}

We recall here some facts about norms on tensor products and topological Hopf algebras. Recall from \cite[Section 17B]{NFA} that given two seminorms $p$ and $p'$ on the vector spaces $V$ and $W$ respectively, the \emph{tensor product seminorm} $p\otimes p'$ on $V\otimes_L W$ is defined in the following way: for $x\in V\otimes_L W$, we have
$$
p\otimes p'(x):=\inf\Big\{\max_{1\leq i\leq r}p(v_i)\cdot p'(w_i) : x=\sum_{i=1}^r v_i\otimes w_i, v_i\in V, w_i\in W\Big\}.
$$
When $V$ and $W$ are Banach spaces or more generally Fr\'echet spaces, the topology obtained via these tensor product (semi)norms agrees with the inductive and projective tensor product topologies on $V\otimes_L W$ (see \cite[Proposition 17.6]{NFA}). One can then construct the Hausdorff completion $\ten{V}{W}$ of this space, which  will be a Banach space (respectively Fr\'echet space). Moreover, if $V$ and $W$ are Hausdorff, so is $V\otimes_L W$.

Then $\widehat{\otimes}_L$ is a monoidal structure on the categories of $L$-Banach spaces and $L$-Fr\'echet spaces. Note that this construction is functorial, so that two continuous linear maps $f: V\to W$ and $g:X\to Y$ induce a continuous linear map $f\widehat{\otimes}g:\ten{V}{X}\to\ten{W}{Y}$.

\begin{definition} An \emph{$L$-Banach coalgebra}, respectively \emph{$L$-Fr\'echet coalgebra}, is a coalgebra object in the monoidal category of $L$-Banach spaces, respectively $L$-Fr\'echet spaces. In other words it is a Banach, respectively Fr\'echet, space $C$ equipped with continuous linear maps $\Delta: A\to \ten{A}{A}$ and $\varepsilon: A\to L$ which satisfy the usual axioms:
$$
(\Delta\widehat{\otimes}\id)\circ\Delta=(\id\widehat{\otimes}\Delta)\circ\Delta, \quad (\id\widehat{\otimes}\varepsilon)\circ\Delta=(\varepsilon\widehat{\otimes}\id)\circ\Delta=\id.
$$
A morphism of coalgebras is then a continuous linear map $f:C\to D$ such that $\varepsilon_D\circ f=\varepsilon_C$ and $(f\widehat{\otimes} \id)\circ \Delta_C=\Delta_D\circ f$.

An \emph{$L$-Banach Hopf algebra}, respectively \emph{$L$-Fr\'echet Hopf algebra}, is an $L$-Banach, respectively Fr\'echet, algebra $A$ which is also a coalgebra such that $\Delta$ and $\varepsilon$ are algebra homomorphisms, and furthemore $A$ is equipped with a continuous linear map $S:A\to A$, which satisfy the usual axioms for a Hopf algebra:
$$
m\circ (S\widehat{\otimes}\id)\circ\Delta=\iota\circ\varepsilon=m\circ (\id\widehat{\otimes}S)\circ\Delta
$$
where $m:\ten{A}{A}\to A$ and $\iota:L\to A$ denote the multiplication map and the unit in $A$ respectively. A morphism of Hopf algebras is then a continuous algebra homomorphism $f:A\to B$ which is also a morphism of coalgebras, such that $S_B\circ f=f\circ S_A$. 
\end{definition}

\subsection{A monoidal functor}\label{monoidal} We now aim to establish that some of the algebras we've constructed are Hopf algebra objects in the categories of $L$-Banach algebras. We will need the following elementary result:

\begin{lem} Let $M,N$ be two $R$-modules. Then we have canonical isomorphisms
$$
(M/\pi^aM)\otimes_R (N/\pi^a N)\cong (M/\pi^aM)\otimes_R N\cong M\otimes_R (N/\pi^aN)\cong (M\otimes_R N)/\pi^a(M\otimes_R N)
$$
for any $a\geq 1$.
\end{lem}

\begin{proof} By tensoring the short exact sequence
$$
0\to \pi^a M\to M\to M/\pi^a M\to 0
$$
with $N$, we obtain an exact sequence
$$
\pi^a M\otimes_R N\to M\otimes_R N\to M/\pi^a M\otimes_R N\to 0.
$$
Thus, since the image of $\pi^aM\otimes_R N$ in $M\otimes_R N$ equals $\pi^a(M\otimes_R N)$, we see that
$$
(M/\pi^a M)\otimes_R N\cong (M\otimes_R N)/\pi^a(M\otimes_R N).
$$
Similarly $M\otimes_R (N/\pi^aN)\cong (M\otimes_R N)/\pi^a(M\otimes_R N)$ by interchanging $M$ and $N$. Finally, if we tensor the short exact sequence
$$
0\to \pi^a N\to N\to N/\pi^a N\to 0
$$
with $M/\pi^aM$, we obtain an exact sequence
$$
(M/\pi^aM)\otimes_R\pi^a N\to (M/\pi^aM)\otimes_R N\to (M/\pi^aM)\otimes_R (N/\pi^a N)\to 0
$$
where the left hand side map clearly has image 0. Thus we get the required isomorphism.
\end{proof}

\begin{prop} Let $M$ and $N$ be torsion-free $R$-modules. Then there is a canonical isomorphism of $L$-Banach spaces
$$
\ten{\widehat{M_L}}{\widehat{N_L}}\cong \widehat{(M\otimes_R N)_L}.
$$
Moreover when $M$ and $N$ are $R$-algebras, this map is an algebra isomorphism. In particular, $M\mapsto\widehat{M_L}$ is a monoidal functor between the category of torsion-free $R$-modules and the category of $L$-Banach spaces.
\end{prop}

\begin{proof} Note that $\widehat{M_L}\otimes_L \widehat{N_ L}\cong (\widehat{M}\otimes_R \widehat{N})\otimes_R L$ and, by the Lemma, we have natural isomorphisms
\begin{align*}
(\widehat{M}\otimes_R \widehat{N})/\pi^a(\widehat{M}\otimes_R \widehat{N})&\cong \widehat{M}/\pi^a \widehat{M}\otimes_R   \widehat{N}/\pi^a \widehat{N}\\
& \cong M/\pi^aM\otimes_R N/\pi^a N\\
& \cong(M\otimes_R N)/\pi^a(M\otimes_R N)
\end{align*}
for all $a\geq 1$. Thus we see that $\widehat{M\otimes_R N}$ is canonically isomorphic to the $\pi$-adic completion of $\widehat{M}\otimes_R \widehat{N}$. Hence we see that $\widehat{(M\otimes_R N)_L}$ is the completion of $\widehat{M_L}\otimes_L \widehat{N_ L}$ with respect to the $\pi$-adic topology on $\widehat{M}\otimes_R \widehat{N}$. By \cite[Lemma 17.2]{NFA}, the latter topology is the same as the tensor product topology on $\widehat{M_L}\otimes_L \widehat{N_ L}$, and so we get the result.

In the case where $M=A$ and $N=B$ are algebras, it is clear from the above that the isomorphism preserves the algebra structure.
\end{proof}

We introduce the following notation: write $\Oqhat:=\widehat{(\A)_L}$.

\begin{cor} The Banach algebras $\Oqhat$ and $\Uqnhat$ $(n\geq 0)$ are $L$-Banach Hopf algebras.
\end{cor}

\begin{proof}
This follows immediately from the Proposition since monoidal functors preserve Hopf algebra objects.
\end{proof}

\begin{example} When $G=\text{SL}_2$ i.e when $\g=\mathfrak{sl_2}$, we can give an explicit description of $\Oqhat$. In that case the only fundamental representation of $U_q$ is two dimensional with basis $v_1, v_2$ such that
$$
Ev_1=0=Fv_2\quad Ev_2=v_1\quad Fv_1=v_2\quad Kv_1=q^{\frac{1}{2}}v_1\quad Kv_2=q^{\frac{-1}{2}}v_2.
$$
The matrix coefficients with respect to that basis are denoted by $x_{11}, x_{12}, x_{21}, x_{22}$ and they generate $\Oq$. As is customary we denote these generators by $a, b, c$ and $d$ respectively. The complete set of relations for $\Oq$ is given by
\begin{align*}
& ab=qba, \quad ac=qca, \quad bc=cb,\quad bd=qdb,\\
& cd=qdc,\quad ad-da=(q-q^{-1})bc, \quad ad-qbc=1.
\end{align*}
(see \cite[Theorem I.7.16]{BroGoo02}).

So in this case $\A$ is the $R$-algebra generated by $a, b, c, d$. By the proof of \cite[Lemma 1.1]{DeConcini94} we see that $\A$ is a free $R$-module and
$$
\Ss=\{a^lb^mc^s, b^mc^sd^t : l,m,s\geq 0 \text{ and } t>0\}
$$
is an $R$-basis of $\A$. Concretely, one can identify $\Oqhat$ as the ring
\begin{align*}
\Oqhat=\Big\{\sum_{l,m,s\geq 0}\lambda_{lms}a^lb^mc^s+\sum_{\substack{p,t\geq 0\\ r>0}}\mu_{ptr}b^pc^td^r : &\abs{\lambda_{lms}}\rightarrow 0 \text{ as } l+m+s\rightarrow\infty\\
&\text{and } \abs{\mu_{ptr}}\rightarrow 0 \text{ as } p+t+r\rightarrow\infty\Big\}.
\end{align*}
This is an $L$-Banach algebra with norm
$$
\norm{\sum\lambda_{lms}a^lb^mc^s+\sum\mu_{ptr}b^pc^td^r}:=\sup_{l,m,s,p,t,r}\{\lambda_{lms}, \mu_{ptr}\}.
$$
\end{example}

We will later give an explicit description of $\Uqnhat$ for $n$ large enough.

\subsection{Hopf algebra structure of $\wideparen{U_q}$}\label{Uqhopf} We recall a few standard facts about Fr\'echet spaces (see e.g. \cite[Section 3]{SchTeit03}). Let $V$ be a Fr\'echet space whose topology is given by a family $p_1\leq p_2\leq\ldots\leq p_n\leq\ldots$ of seminorms. For each $n$ the seminorm $p_n$ induces a norm on the quotient $V/\{v\in V: p_n(v)=0\}$. The completion of this normed space is a Banach space, which we denote by $V_{p_n}$. The identity on $V$ induces continuous linear maps $V_{p_{n+1}}\to V_{p_n}$ for all $n$. Then the natural map
$$
V\to \varprojlim V_{p_n}
$$
is an isomorphism of locally convex $L$-spaces. When $V$ is a Fr\'echet algebra, and all the seminorms $p_n$ are algebra seminorms, then this map is an $L$-algebra isomorphism.

\begin{prop}\emph{(\cite[Proposition 1.1.29]{Emerton})} Let $V$ and $W$ be $L$-Fr\'echet spaces whose topologes are defined by families of seminorms $p_1\leq p_2\leq\ldots\leq p_n\leq\ldots$ and $p'_1\leq p'_2\leq\ldots\leq p'_n\leq\ldots$ respectively. Then we have a canonical isomorphism of $L$-Fr\'echet spaces
$$
V\widehat{\otimes}_L W\cong\varprojlim V_{p_n}\widehat{\otimes}_L W_{p'_n}.
$$
When $V$ and $W$ are Fr\'echet algebras and all the seminorms are algebra seminorms, this is an algebra isomorphism.
\end{prop}

Using this result, we can prove:

\begin{thm} The functor $M\mapsto \wideparen{M_L}$ on the category of torsion-free filtered $R$-modules is monoidal. In particular the Fr\'echet algebra $\wideparen{U_q}$ is an $L$-Fr\'echet Hopf algebra.
\end{thm}

\begin{proof} From the above Proposition we see that for any two torsion-free filtered $R$-modules $M$ and $N$, there is a canonical isomorphism of $L$-Fr\'echet spaces
$$
\wideparen{M_L}\widehat{\otimes}_L \wideparen{N_L}\cong \varprojlim \widehat{M_{n,L}}\widehat{\otimes}_L \widehat{N_{n,L}}
$$
which is an algebra isomorphism when $M$ and $N$ are $R$-algebras. Now, the first result follows by Proposition \ref{monoidal} and Lemma \ref{Schneiderspropn}. The fact that $\wideparen{U_q}$ is an $L$-Fr\'echet Hopf algebra now follows because monoidal functors preserve Hopf algebra objects, and $U$ is a filtered Hopf algebra, meaning that $\Delta$, $\varepsilon$ and $S$ are filtered maps (where for $\varepsilon$ we give $R$ the trivial filtration).
\end{proof}

\subsection{Hopf algebra structure of $\Oqcap$}\label{OqHopf}

We know that $\A$ is a Hopf algebra, however the corresponding Hopf algebra maps are not all filtered $R$-module homomorphisms on $\A$, so we can't immediately deduce from our previous methods that $\wideparen{\Oq}$ has a Hopf algebra structure. On the other hand, we see from equation (\ref{Hopfrelations}) in \ref{prelimonOq} that the counit restricted to $\A$ is a filtered $R$-map $\A\to R$ and so gives rise to a map $\wideparen{\epsilon}:\wideparen{\Oq}\to L$. For the antipode and comultiplication, we can ``shift'' the deformations to make things work.

Indeed, from (\ref{Hopfrelations}) we have $\Delta(F_n\A)\subseteq F_n\A\otimes_RF_n\A$ for all $n\geq 0$. But then it follows that for all $n\geq 0$ we have
$$
\Delta((\A)_{2n})\subseteq (\A)_n\otimes_R(\A)_n.
$$
Taking $\pi$-adic completions we see that $\Delta$ induces maps
$$
\widehat{\Delta}_n: \widehat{(\A)_{2n,L}}\to \ten{\widehat{(\A)_{n,L}}}{\widehat{(\A)_{n,L}}}.
$$
Taking inverse limits we obtain a map
$$
\wideparen{\Delta}: \Oqcap\to \ten{\Oqcap}{\Oqcap}
$$
We now move to the antipode. It's not necessarily clear that it's a filtered map on $\A$, so we let
$$
d=\max_{1\leq i\leq r}\{\min\{t: S(x_i)\in F_t\A\}\}.
$$ 
It follows that $S((\A)_{nd})\subseteq (\A)_n$ for all $n\geq 0$. Taking $\pi$-adic completions we see that $S$ induces maps
$$
\widehat{S}_n: \widehat{(\A)_{nd,L}}\to \widehat{(\A)_{n,L}}.
$$
Taking inverse limits we obtain a map
$$
\wideparen{S}: \Oqcap\to \Oqcap.
$$
We see that the maps $\wideparen{\epsilon}$, $\wideparen{S}$ and $\wideparen{\Delta}$ make $\wideparen{\Oq}$ into a Hopf algebra, as desired, since all the Hopf algebra relations are satisfied on the dense subspace $\Oq$.

\begin{remark} Note that the above shifts really are to be expected. Indeed, for example in the case $G=\text{SL}_n(L)$, the algebra we construct is meant to be a quantum analogue of the global sections of the structure sheaf of the analytification of $G$. If $\Of$ denotes the coordinate algebra of SL$_n(R)$, this ring of global sections is given by the inverse limit of the Banach algebras $\widehat{\Of_{m,L}}$, which correspond to the functions on $G$ which are analytic on SL$_n(\pi^{-m}R)$. For $m>0$, since that subset of $G$ is not a subgroup, the algebra $\widehat{\Of_{m,L}}$ is not a Hopf algebra. On the other hand matrix multiplication defines a map
$$
\text{SL}_n(\pi^{-m}R)\times \text{SL}_n(\pi^{-m}R)\to \text{SL}_n(\pi^{-2m}R)
$$
which induces a map $\Delta:\widehat{\Of_{2m,L}}\to\ten{\widehat{\Of_{m,L}}}{\widehat{\Of_{m,L}}}$. Our quantum situation very much mirrors this.
\end{remark}

\section{Fr\'echet--Stein structures}\label{FrechetStein}

\subsection{Fr\'echet--Stein algebras}\label{analytificationfunctor} We start with a definition.

\begin{definition}
Following \cite[Section 3]{SchTeit03} we say that an $L$-algebra $\U$ is \emph{$L$-Fr\'echet-Stein} if there is a tower $\U_0\leftarrow \U_1 \leftarrow \U_2\leftarrow\cdots$ of Noetherian $L$-Banach algebras such that

\begin{enumerate} 
\item $\U_n$ is a flat $\U_{n+1}$-module for all $n\geq 0$,
\item $\U\cong \varprojlim \U_n$.
\end{enumerate}

\end{definition}

Our main aim is going to prove that the algebras $\Oqcap$ and $\Uqcap$ are Fr\'echet-Stein. The main difficulty in proving that an algebra satisfies the above definition is to show that the flatness condition in (i) holds. To do this we rely on two known results. The first one, due to Emerton, is the following:

\begin{prop}\emph{(\cite[Proposition 5.3.10]{Emerton})}
Suppose that $A$ is a left Noetherian $R$-algebra, $\pi$-adically separated, $\pi$-torsion free, and suppose that $B$ is an $R$-subalgebra of $A_L$ which contains $A$. Suppose $B$ is equipped with an exhaustive $R$-algebra filtration $(F_\cdot)$ satisfying $F_0B=A$ and such that $\gr B$ is finitely generated as an $A$-algebra by central elements. Then $\widehat{A_L}$ and $\widehat{B_L}$ are left Noetherian and $\widehat{B_L}$ is right flat over $\widehat{A_L}$.
\end{prop}

The second one is due to Ardakov and Wadsley, and is using a certain class of deformable algebras as well the functor we defined in \ref{Mcap}.

\begin{thm}\emph{(\cite[Theorem 6.7]{Wadsley2})}
Let $U$ be a deformable $R$-algebra such that $\gr U$ is commutative and Noetherian. Then $\wideparen{U_L}$ is a Fr\'echet-Stein algebra.
\end{thm}

The issue with these methods is that the statements both involve some commutativity or centralness conditions that will not hold in the quantum setting. Therefore, in this section, we will prove certain non-commutative, or quantum, versions of these results.

\subsection{Fr\'echet completions of deformable $R$-algebras}\label{Frechetcompl}

We first generalise Theorem\ref{analytificationfunctor}. The proofs from \cite[Section  6.5 \& 6.6]{Wadsley2} go through with only minor changes. Throughout, we will make the following assumptions:
\begin{enumerate}
\item $U$ is a deformable $R$-algebra such that $\gr_0 U$ and $\gr U$ are Noetherian;
\item there are elements $x_1, \ldots, x_r\in U$ such that
$$
F_iU=F_0U\cdot\{x_1^{\alpha_1}\cdots x_r^{\alpha_r}: \sum_{j=1}^r \alpha_jd_j\leq i\}
$$
for each $i\geq 0$, where $d_j=\deg{x_j}$, so that then $\gr U$ is finitely generated over $\gr_0 U$ by the symbols $x_1, \ldots, x_r\in U$; and
\item the sequence $\overline{\pi x_1},\ldots, \overline{\pi x_r}$, where $\overline{\pi x_i}$ denotes the image of $\pi x_i$ in $U_1/\pi U_1$, is polynormal.
\end{enumerate}
Note that (i)-(iii) hold when $U$ is a deformable $R$-algebra such that $\gr U$ is commutative and Noetherian by the proofs in \cite[Section 6.5 \& 6.6]{Wadsley2}.

\begin{lem}
If $U$ satisfies (i) and (ii) as above, then so does $U_n$ for all $n\geq 1$.
\end{lem}

\begin{proof}
This is a straightforward application of Lemma \ref{deformablealg}(i): (i) follows immediately because $\gr U_n\cong \gr U$ and (ii) follows because $$
F_iU_n=F_0U\cdot\{(\pi^{nd_1}x_1)^{\alpha_1}\cdots (\pi^{nd_r}x_r)^{\alpha_r}: \sum_{j=1}^r \alpha_jd_j\leq i\}
$$
from which we see that $\gr U_n$ is generated by the symbols of $\pi^{nd_1}x_1,\ldots, \pi^{nd_r}x_r$ over $\gr_0 U_n$.
\end{proof}

\begin{prop} Let $U$ be a deformable $R$-algebra satisfying condition (ii) above, and consider the ideal $I:=U_1\cap\pi U$.
\renewcommand{\labelenumi}{(\alph{enumi})}
\begin{enumerate}
\item The subspace filtration on $U_1$ of the $\pi$-adic filtration on $U$ and the $I$-adic filtration on $U_1$ are topologically equivalent; and
\item $I$ is generated by $\pi$ and $\pi^{d_j}x_j$ for $1\leq j\leq n$.
\end{enumerate}
\end{prop}

\begin{proof} It is clear from the definition of $I$ that
$$
\pi\in I\quad \text{and}\quad \pi^{d_j}x_j\in I\quad \text{for all} \quad 1\leq j\leq n.
$$
Let $d_0:=1$. It follows from condition (ii) that $\pi^iF_iU$ is generated as an $F_0U$-module by monomials of the form
\begin{equation}\label{equation1}
(\pi^{d_0})^{\alpha_0}(\pi^{d_1}x_1)^{\alpha_1}\cdots (\pi^{d_n}x_n)^{\alpha_n}
\end{equation}
where $\alpha_j\geq 0$ for all $j=0,\ldots, n$ and $\sum_{j=0}^n\alpha_jd_j=i$. For any integer $t\geq 0$ and $i\geq t\max{d_j}$, we have $(\sum_{j=0}^n\alpha_j)\max{d_j}\geq \sum_{j=0}^n\alpha_jd_j=i\geq t\max{d_j}$, so
$$
(\pi^{d_0})^{\alpha_0}(\pi^{d_1}x_1)^{\alpha_1}\cdots (\pi^{d_n}x_n)^{\alpha_n}\in I^t
$$
since $\pi\in I$ and $\pi^{d_j}x_j\in I$ for all $1\leq j\leq m$. Hence by Lemma \ref{deformablealg}(ii) we have
$$
U_1\cap \pi^{t\max{d_j}}U=\sum_{i\geq t\max{d_j}}\pi^i F_iU\subseteq I^t\subseteq U_1\cap\pi^tU
$$
since $I$ is an $F_0U$-submodule of $U$, thus proving (a).

For (b), by Lemma \ref{deformablealg}(ii) we have $I=\sum_{i\geq 1}\pi^iF_iU$. But we know from (\ref{equation1}) above that, for $i\geq 1$, $\pi^iF_iU$ is generated as an $F_0U$-module by elements which are in the ideal generated by $\pi$ and $\pi^{d_j}x_j$ for $1\leq j\leq n$. The result follows.
\end{proof}

We can now prove our version of \cite[Theorem 6.6]{Wadsley2}. Their proof goes through unchanged except for our use of condition (iii) which replaces their commutativity constraint.

\begin{thm} Let $U$ be a deformable $R$-algebra satisfying conditions (i)-(iii). Then $\widehat{U_{L}}$ is flat over $\widehat{U_{1, L}}$.
\end{thm}

\begin{proof} Since $\widehat{U_{1, L}}=\widehat{U_1}\otimes_R L$, it is enough to show that $\widehat{U_{L}}$ is flat as a module over $\widehat{U_1}$. By the Proposition, the $I$-adic completion $B$ of $U_1$ is isomorphic to the closure of the image of $U_1$ in $\widehat{U}$. Hence we have natural maps $\widehat{U_1}\to B\to \widehat{U_L}$. Observe that $B$ is $\pi$-adically complete by the proof of \cite[Theorem VIII.5.14]{ZarSam}, noting that ideals in $B$ are $I$-adically closed by \cite[Theorem II.2.1.2, Proposition II.2.2.1]{Zariskian}.

We observe that $B/\pi B$ is the $I/\pi U_1$-adic completion of $U_1/\pi U_1$. From Proposition \ref{Frechetcompl}(ii), the ideal $I/\pi U_1$ is generated by $\overline{\pi^{d_j}x_j}$ for $1\leq j\leq n$. Hence it follows from condition (iii) and \cite[Proposition D.V.1 \& Remark D.V.2]{gradedring} that $I/\pi U_1$ has the Artin-Rees property. Thus we have that $B/\pi B$ is flat over $U_1/\pi U_1$ by \cite[Property V.8)iii), page 301]{gradedring}.

We now filter both $\widehat{U_1}$ and $B$ $\pi$-adically. Since $U_1$ is $\pi$-torsion free, we have $\gr\widehat{U_1}\cong (U_1/\pi U_1)[t]$. In a similar way, since $B$ is isomorphic to a subring of $\widehat{U}$ and so has no $\pi$-torsion, we have $\gr B\cong (B/\pi B)[t]$. Hence $\gr B$ is flat over $\gr\widehat{U_1}$. But this implies that $B$ is a flat $\widehat{U_1}$-module by \cite[Proposition 1.2]{SchTeit03}, since both $\widehat{U_1}$ and $B$ are $\pi$-adically complete.

We now consider the subspace filtration on $U_1$ induced from the $\pi$-adic filtration on $U$. We have $\gr U\cong \overline{U}[t]$ where $t=\gr \pi$ and $\overline{U}=U/\pi U$ has degree zero. Lemma \ref{deformablealg}(ii) implies that the image of $\gr U_1$ inside $\gr U$ is $\oplus_{j\geq 0}t^j \overline{F_jU}$ where $\overline{F_jU}$ denotes the image of $F_jU$ in $\overline{U}$. Note that $\gr U_1$ is Noetherian by \cite[Corollary 1.3]{BroGoo97} and conditions (i) and (iii) since it is generated by the $t^{d_i}\overline{x_i}$. Now, as the quotient filtration $\overline{F_jU}$ on $\overline{U}$ is exhaustive, the localisation of this image obtained by inverting $t$ is $\overline{U}[t, t^{-1}]$. But $B$ is the completion of $U_1 $ so
$$
(\gr B)_t=(\gr U_1)_t=\overline{U}[t, t^{-1}]=\gr \widehat{U_{L}}.
$$
Hence $\gr \widehat{U_{L}}$ is flat over $\gr B$. We can then invoke \cite[Proposition 1.2]{SchTeit03} again to conclude that $\widehat{U_{L}}$ is flat over $B$.
\end{proof}

\subsection{Theorem}\label{defFreSt} \emph{Let $U$ be a deformable $R$-algebra satisfying assumptions (i)-(iii), such that $U_n$ satisfies (iii) for all $n\geq 0$. Then $\wideparen{U_{L}}$ is a Fr\'echet-Stein algebra.}

\begin{proof}
By Lemma \ref{Frechetcompl} each $U_n$ satisfies conditions (i)-(iii). Now since $(U_n)_1=U_{n+1}$ by Lemma \ref{deformablealg}, we have by the Theorem that $\widehat{U_{n, L}}$ is a flat $\widehat{U_{n+1, L}}$-module. Moreover, each $\widehat{U_{n, L}}$ is Noetherian because $\gr U$ is Noetherian.
\end{proof}

We now turn to the important notion of a coadmissible module:

\begin{definition}[{\cite[Section 3]{SchTeit03}}]\label{FS1} Let $\U=\varprojlim \U_n$ be a Fr\'echet-Stein algebra. Then a $\U$-module $\M$ is called \emph{coadmissible} if $\M\cong \varprojlim \M_n$ where, for each $n\geq 0$, $\M_n$ is a finitely generated $\U_n$-module and $\U_n\otimes_{\U_{n+1}}\M_{n+1}\cong \M_n$. The full subcategory of coadmissible modules is denoted by $\C(\U)$.
\end{definition}

Note that if $\M$ is a coadmissible module, then each $\M_n$ naturally inherits the structure of a Banach $\U_n$-module, and so $\M$ naturally has the structure of a Fr\'echet space.

We summarise below the facts we'll need:

\begin{prop}[{\cite[Lemma 3.6 \& Corollaries 3.1, 3.4 \& 3.5]{SchTeit03}}]\label{FS2} Let $\U$ be a Fr\'echet-Stein algebra and let $\M$ be a coadmissible $\U$-module.
\begin{enumerate}
\item For each $n\geq 0$, $\M_n\cong \U_n\otimes_{\U} \M$.
\item The category $\C(\U)$ is an abelian subcategory of the category of all $\U$-modules; it is closed under direct sums and contains the finitely presented $\U$-modules.
\item Let $\Nn$ be a submodule of $\M$. Then the following are equivalent:
\begin{enumerate}
\item $\Nn$ is coadmissible;
\item $\M/\Nn$ is coadmissible; and
\item $\Nn$ is closed in the above Fr\'echet topology.
\end{enumerate}
\item A sum of two coadmissible submodules of $\M$ is coadmissible.
\item Any finitely generated submodule of $\M$ is coadmissible.
\item Any module map between two coadmmissible module is strict with closed image.
\end{enumerate}
\end{prop}

The proof of the next result is essentially the proof of the first part of \cite[Theorem 4.11]{SchTeit03} (see also \cite[Theorem 4.3.3]{Schmidt2}) but we reproduce it here for the convenience of the reader.

\begin{cor}\label{flat} Let $U$ be a deformable $R$-algebra satisfying assumptions (i)-(iii), such that $U_n$ satisfies (iii) for all $n\geq 0$. Then the natural map $U_L\to\wideparen{U_L}$ is flat.
\end{cor}

\begin{proof}
We show right flatness, the proof of left flatness being completely analogous. Since $\pi$ is central, for every $n\geq 0$ the ideal $\pi U_n$ in $U_n$ satisfies the Artin-Rees property and thus $\widehat{U_n}$ is flat over $U_n$ by \cite[Proposition D.V.1 \& Property V.8)iii), page 301]{gradedring}. Hence it follows that $U_L\to \widehat{U_{n,L}}$ is flat for every $n\geq 0$. By the Theorem we know that $\wideparen{U_L}$ is Fr\'echet-Stein. It will suffice to show that for a left ideal $I\subset U_L$, the map $\wideparen{U_L}\otimes_{U_L}I\to \wideparen{U_L}$ is injective. But now, $I$ is finitely generated and in fact finitely presented since $U_L$ is Noetherian. Thus $\wideparen{U_L}\otimes_{U_L}I$ is finitely presented as well, and so coadmissible. Thus we have isomorphisms
$$
\wideparen{U_L}\otimes_{U_L}I\cong \varprojlim \left(\widehat{U_{n,L}}\otimes_{\wideparen{U_L}}(\wideparen{U_L}\otimes_{U_L}I) \right)\cong \varprojlim (\widehat{U_{n,L}}\otimes_{U_L}I).
$$
Now as $\widehat{U_{n,L}}$ is flat over $U_L$ for every $n$, it follows that $\widehat{U_{n,L}}\otimes_{U_L}I\to \widehat{U_{n,L}}$ is injective. The result then follows since projective limits preserve injections.
\end{proof}

\subsection{Emerton's result}\label{Emerton} When it is not known whether the algebras we have at hand are deformable, we instead rely on techniques inspired from Emerton's result to prove that their completions are Fr\'echet-Stein. Again, the arguments from \cite[5.3.5-5.3.10]{Emerton} follow through with only minor changes. They mainly rely on some general lemmas that we do not write out here but reference throughout the proof.

\begin{prop}
Suppose that $A$ is a left Noetherian $R$-algebra, $\pi$-adically separated, $\pi$-torsion free, and suppose that $B$ is an $R$-subalgebra of $A_L$ which contains $A$. Suppose $B$ is equipped with an exhaustive $R$-algebra filtration $(F_\cdot)$ satisfying $F_0B=A$ and such that $\gr^F B$ is a $q$-commutative $A$-algebra. Then $\widehat{A_L}$ and $\widehat{B_L}$ are left Noetherian and $\widehat{B_L}$ is right flat over $\widehat{A_L}$.
\end{prop}

\begin{proof}
Note that $\widehat{A}$ is left Noetherian because $A$ is left Noetherian, hence so is $\widehat{A_L}$. Furthermore, $\gr B$ is left Noetherian by Lemma \ref{prelimonOq}. Now, following \cite{Emerton}, for any left $A$-submodule $M$ of $A_L$, we let $\iota_M:\widehat{A}\otimes_AM\to \widehat{A_L}$ be the natural map induced from the multiplication in $\widehat{A_L}$, and we let $C$ denote the image of $\iota_B$. By \cite[Corollary 5.3.6]{Emerton} $C$ is an $R$-subalgebra of $\widehat{A_L}$. Let $G_iC$ denote the image of $\iota_{F_iB}$. By \cite[Lemma 5.3.5]{Emerton}, $G_iC$ is equal to $F_iB+\widehat{A}$ and $C=B+\widehat{A}$, and so we see that $(G'_\cdot)$ is an exhaustive algebra filtration on $C$ such that $G'_0C=\widehat{A}$. Now, by \cite[Lemma 5.3.5]{Emerton}, $F_{i-1}B=A_L\cap G_{i-1}C$ for all $i\geq 1$ and so it follows that $F_{i-1}B=F_iB\cap(F_{i-1}B+\widehat{A})$. Hence the natural map $\gr^F_iB\to\gr^G_iC$ induced by $\iota_{F_iB}$ is an isomorphism. Thus we deduce from our assumptions that the associated graded ring $\gr^{G'}C$ is a $q$-commutative $\widehat{A}$-algebra. Therefore by Lemma \ref{prelimonOq} we have that $\gr^{G'}C$ is left Noetherian, hence so is $C$.

The fact that $\widehat{B_L}$ is right flat over $\widehat{A_L}$ now follows easily. Indeed, since $C=B+\widehat{A}$ we see that $C_L=\widehat{A_L}$. Moreover $\widehat{B_L}\cong \widehat{C_L}$ by \cite[Lemma 5.3.8]{Emerton}. But the ideal generated by $\pi$ satisfies the Artin-Rees property as $\pi$ is central, and so $\widehat{C}$ is right flat over $C$ as $C$ is left Noetherian. Tensoring over $R$ with $L$, we therefore see that $\widehat{B_L}\cong \widehat{C_L}$ is right flat over $\widehat{A_L}=C_L$. %Left flatness will follow by the same argument applied to $B^\text{op}$.
\end{proof}

\subsection{A PBW type $R$-basis}\label{PBW} In order to apply the previous results to $\Uqcap$, it will be useful to find certain bases of the algebras $U_n$. These will in turn allow us to get an explicit description of $\Uqcap$.

Let $\U$ be the $R$-submodule of $U_q$ spanned by all monomials $M_{\boldsymbol{r}, \boldsymbol{s}, \lambda}$, which is free by the PBW theorem. The height filtration on $U_q$ induces a filtration on $\U$. Explicitly, we define $F_i\U$ to be the $R$-span of the monomials $M_{\boldsymbol{r}, \boldsymbol{s}, \lambda}$ such that $\htt(M_{\boldsymbol{r}, \boldsymbol{s}, \lambda})\leq i$. We want to deform this module and eventually obtain an algebra. For each $n\geq 0$, the $R$-module $\U_n$ is just the $R$-span of all $\pi^{n\htt(M_{\boldsymbol{r}, \boldsymbol{s}, \lambda})}M_{\boldsymbol{r}, \boldsymbol{s}, \lambda}$, or in other words the $R$-span of the monomials 
$$
(\pi^{n\htt(\beta_1)}F_{\beta_1})^{r_1}\cdots (\pi^{n\htt(\beta_N)}F_{\beta_N})^{r_N}K_\lambda (\pi^{n\htt(\beta_1)}E_{\beta_1})^{s_1}\cdots (\pi^{n\htt(\beta_N)}E_{\beta_N})^{s_N}.
$$
We let $m$ be the least integer such that
$$
\frac{\pi^{2m}}{q_i-q_i^{-1}}\in R\quad \text{for all } 1\leq i\leq n.
$$
Hence for all $n\geq m$, we have
$$
(\pi^nE_{\alpha_i})(\pi^nF_{\alpha_i})-(\pi^nF_{\alpha_i})(\pi^nE_{\alpha_i})\in R[K_\lambda :\lambda\in P]
$$
and so the generators of $U_n$ satisfy relations which can be expressed as an $R$-linear combination of them.

\begin{thm}
The $R$-module $\U_n$ is equal to $U_n$ for all $n\geq m$, and so is an $R$-algebra.
\end{thm}

We start preparing for the proof the Theorem. For all $n\geq 0$, we let $U_n^+$ be the positive part of $U_n$, i.e the $R$-subalgebra of $U_q$ generated by the $\pi^nE_{\alpha_i}$'s. It is the $n$-th deformation of $U^+$ with respect to the filtration given by assigning every $E_{\alpha_i}$ degree 1. We also define $\U_n^+$ to be the $R$-submodule of $\U_n$ spanned by all monomials of the form
$$
(\pi^{n\htt(\beta_1)}E_{\beta_1})^{s_1}\cdots (\pi^{n\htt(\beta_N)}E_{\beta_N})^{s_N}.
$$
It is the $n$-th deformation of $\U^+$ with respect to the height filtration. We also define $U_n^-$ and $\U_n^-$ by applying $\omega$ to the positive parts.

Since $p=\text{char}(k)>3$, we see that the braid group action from \ref{PrelimonUq} preserves $U$ and so $E_{\beta_j}$ lies in $U$ for all $1\leq j\leq N$. Since the automorphism $\omega$ preserves $U$, we see that the $F_{\beta_j}$'s also belong to $U$, and hence that $\U\subset U$. To obtain that the $\pi^{n\htt{\beta_j}}E_{\beta_j}$'s actually lie in $U_n^+$ for every $n\geq 0$, we adapt \cite[Lemma 8.19 and Proposition 8.20]{Jantzen} to our situation. The same proofs go through with only minor changes. Before that, we establish the following notation: for a sequence $J=\{\alpha_{i_1}, \ldots, \alpha_{i_j}\}$ of simple roots, we write $E_J$ for the product $E_{\alpha_{i_1}}\cdots E_{\alpha_{i_j}}$.

\begin{lem}
Let $w\in W$ and $\alpha$ be a simple root. Suppose $w\alpha>0$ and write $w\alpha=\sum_{i=1}^{n}m_i\alpha_i$. Then $T_w(E_\alpha)$ is an $R$-linear combination of words all of the form $E_J$ where $J$ is a finite sequence of simple roots such that each root $\alpha_i$ occurs in $J$ with multiplicity $m_i$.
\end{lem}

\begin{proof}
We first prove the result in a particular case.

\begin{claim}
Suppose $\beta\neq \alpha$ is another simple root and assume $w$ is in the subgroup of $W$ generated by $s_\alpha$ and $s_\beta$. Then the result holds.
\end{claim}

\begin{proof}[Proof of claim]
We are reduced to a rank 2 case-by-case analysis. If $w=1$ the result is trivial so assume $w\neq 1$. Denote by $m$ the order of $s_\alpha s_\beta$. We have $m=2, 3, 4$ or $6$.

If $m=2$ then $w=s_\beta$ and $T_w(E_\alpha)=E_\alpha$. If $m=3$ then
$$
w\in\{s_\beta, s_\alpha s_\beta\}.
$$
If $m=4$ then
$$
w\in\{s_\beta, s_\alpha s_\beta, s_\beta s_\alpha s_\beta\}.
$$
If $m=6$ then
$$
w\in \{s_\beta, s_\alpha s_\beta, s_\beta s_\alpha s_\beta, s_\alpha s_\beta s_\alpha s_\beta, s_\beta s_\alpha s_\beta s_\alpha s_\beta\}.
$$
Hence in all cases we see that $T_w(E_\alpha)$ is just one of the root vectors that arise in the PBW basis for the case where $\g$ has rank 2. The result then follows by the formulae in \cite[Appendix,(A1)-(A3)]{DeConProc} using our assumptions on $p$.
\end{proof}

We now use induction on $\ell(w)$. If $\ell(w)=0$ then $T_w=1$ and the result is trivial. So assume that $\ell(w)>0$. Hence there exists a simple root $\beta$ such that $w\beta<0$ (and so $\alpha\neq\beta$). By standard facts about Coxeter groups (see \cite{Humphreys2}), we have a decomposition $w=w'w''$ where $w''$ lies in the subgroup of $W$ generated by $s_\alpha$ and $s_\beta$ such that $w'\beta>0$ and $w'\alpha>0$. Then $\ell(w)=\ell(w')+\ell(w'')$ so that $T_w=T_{w'}T_{w''}$. Moreover since $w\alpha>0$ and $w\beta <0$ it follows that $w''\alpha>0$ and $w''\beta<0$. In particular $w''\neq 1$. By the claim we have that $T_{w''}(E_{\alpha})$ is an $R$-linear combination of words all of the form $E_{J''}$ where $J''$ is a finite sequence of simple roots only involving $\alpha$ and $\beta$ such that they appear with the appropriate multiplicities. By induction hypothesis, we also have that $T_{w'}(E_\alpha)$ is an $R$-linear combination of words all of the form $E_{J'}$ where $J'$ is a finite sequence of simple roots each simple root appears in $J'$ with the appropriate multiplicity. Similarly, the analogous statement is true for $T_{w'}(E_\beta)$. Now the result follows since $T_w=T_{w'}T_{w''}$.
\end{proof}

\begin{cor}
Fix a reduced expression $w_0=s_{i_1}\cdots s_{i_N}$. For any $1\leq j\leq N$, write $\beta_j=\sum_{i=1}^{n}m_{ij}\alpha_i$. Then $E_{\beta_j}$ is an $R$-linear combination of words all of the form $E_J$ where $J$ is a finite sequence of simple roots such that each root $\alpha_i$ occurs in $J$ with multiplicity $m_{ij}$ (and so $J$ has length $\htt{\beta_j}$).
\end{cor}

\begin{proof}
Since $\beta_j:=s_{i_1}\cdots s_{i_{j-1}}(\alpha_{i_j})$ we can write it as $w\alpha$ where $w=s_{i_1}\cdots s_{i_{j-1}}$ and $\alpha=\alpha_{i_j}$.
\end{proof}

In particular, the Corollary implies that, for all $n\geq 0$, $\pi^{n\htt(\beta_j)}E_{\beta_j}\in U_n^+$ for all $1\leq j\leq N$. Similarly $\pi^{n\htt(\beta_j)}F_{\beta_j}\in U_n^-$ for all $j$. Hence we see that $\U^{\pm}\subseteq U^{\pm}$ and that $\U_n\subseteq U_n$ for all $n\geq 0$.

\begin{remark} Although the proof that $E_{\beta_j}\in U^+$ is well-known, we couldn't find a reference for the result about multiplicities so we included the proofs for that.
\end{remark}

\subsection{Proof of Theorem \ref{PBW}}\label{proof_of_PBW} The argument to prove the theorem is the same as in \cite[Theorem 8.24]{Jantzen}, rephrased in our context. We sketch it here. By our choice of $m$, the map $U_m^-\otimes_R U_m^0\otimes_R U_m^+\to U_m$ is surjective, where $U_m^0=R[K_\lambda:\lambda\in P]=RP$. Since the left hand side is a lattice inside $U_q^-\otimes_L U_q^0\otimes_L U_q^+$ and by using the triangular decomposition for $U_q$, we in fact get that this map is an isomorphism. We also clearly have a triangular decomposition $\U_m\cong \U_m^-\otimes_R \U_m^0\otimes_R \U_m^+$ where $\U_m^0=U_m^0$.

Since the automorphism $\omega$ preserves $U_m$, we only have to check that $U_m^+=\U_m^+$ in order to obtain $U_m=\U_m$. In fact we show that $U^+=\U^+$ and that this implies that $U_n^+=\U_n^+$ for every $n\geq 0$.

\begin{prop}
Let $w\in W$ and choose a reduced expression $w=s_{j_1}\cdots s_{j_t}$. Denote by $\U^+[w]$ the $R$-span of all monomials of the form
\begin{equation}\label{monomials}
E_{\beta_1}^{m_1}\cdots E_{\beta_t}^{m_t}
\end{equation}
where $E_{\beta_i}=T_{\alpha_{j_1}}\cdots T_{\alpha_{j_{i-1}}}(E_{\alpha_{j_i}})$ for $1\leq i\leq t$. Then $\U^+[w]$ depends only on $w$, not of the choice of reduced expression.
\end{prop}

\begin{proof}
This is identical to the proof of \cite[Proposition 8.22]{Jantzen}, noting that the rank 2 calculations that they perform all take place inside $U^+$.
\end{proof}

\begin{cor}
We have $U_n^+=\U_n^+$ for every $n\geq 0$. Moreover, the height filtration on $\U^+=U^+$ equals the filtration obtained by assigning every $E_{\alpha_i}$ degree 1.
\end{cor}

\begin{proof}
By the Proposition we see that $\U^+=\U^+[w_0]$ is independent of the choice of reduced expression for $w_0$, and thus is preserved under left multiplication by all the generators $E_{\alpha_i}$ by the proof of \cite[Theorem 8.24]{Jantzen}. Hence $U^+=\U^+$ since $1\in \U^+$.

The height filtration on $U^+$ is an algebra filtration as it is the subspace filtration of an algebra filtration on $U_q^+$. Since all the $E_{\alpha_i}$'s have degree 1 in it, it must contain the filtration where we set $\deg(E_{\alpha_i})=1$. Corollary \ref{PBW} gives the reverse inclusion. Thus we now obtain $U_n^+=\U_n^+$ by taking the $n$-th deformation with respect to this filtration.
\end{proof}

\begin{proof}[Proof of Theorem \ref{PBW}] Put $n=m$ in the previous Corollary to obtain that $U_m=\U_m$. Moreover, by the same proof as in the previous Corollary, we get that the height filtration on $U_m$ equals to filtration obtained by setting $F_0U_m=R[K_\lambda :\lambda\in P]$ and $\deg(E_{\alpha_i})=\deg(F_{\alpha_i})=1$. Hence we get that $U_n=\U_n$ for every $n\geq m$ by deforming.
\end{proof}

\begin{remark}
We see that the only thing stopping $U$ from being equal to $\U$ is the commutator relations between the $E$'s and the $F$'s, which stop the triangular identity as we wrote it from holding in $U$. We can fix this slightly by noticing that we have $U\cong U^-\otimes_R U^0\otimes U^+$ with a slightly different choice of $U^0$: we define it to be the $R$-algebra generated by the $K_\lambda$, $\lambda\in P$, and the elements
$$
[K_{\alpha_i};0]_{q_i}:=\frac{K_{\alpha_i}-K_{\alpha_i}^{-1}}{q_i-q_i^{-1}}
$$
for all $i\geq 0$.
\end{remark}

We can also use Theorem \ref{PBW} to get an explicit description of $\Uqnhat$ for $n\geq m$. Indeed we see that as a topological vector space it is given by the series
$$
\Uqnhat=\left\{\sum_{\boldsymbol{r},\boldsymbol{s},\lambda} a_{\boldsymbol{r},\boldsymbol{s},\lambda}M_{\boldsymbol{r}, \boldsymbol{s}, \lambda} : \abs{\pi^{-n\htt(M_{\boldsymbol{r}, \boldsymbol{s}, \lambda})}a_{\boldsymbol{r},\boldsymbol{s},\lambda}}\to 0 \text{ as } \htt(M_{\boldsymbol{r}, \boldsymbol{s}, \lambda})\to\infty \right\}.
$$
The norm on $\Uqnhat$ is then given by
$$
\norm{\sum_{\boldsymbol{r},\boldsymbol{s},\lambda} a_{\boldsymbol{r},\boldsymbol{s},\lambda}M_{\boldsymbol{r}, \boldsymbol{s}, \lambda}}_n=\sup_{\boldsymbol{r},\boldsymbol{s},\lambda} \abs{\pi^{-n\htt(M_{\boldsymbol{r}, \boldsymbol{s}, \lambda})}a_{\boldsymbol{r},\boldsymbol{s},\lambda}}.
$$
One can then similarly describe $\Uqcap$:
$$
\Uqcap=\left\{\sum_{\boldsymbol{r},\boldsymbol{s},\lambda} a_{\boldsymbol{r},\boldsymbol{s},\lambda}M_{\boldsymbol{r}, \boldsymbol{s}, \lambda} : \abs{\pi^{-n\htt(M_{\boldsymbol{r}, \boldsymbol{s}, \lambda})}a_{\boldsymbol{r},\boldsymbol{s},\lambda}}\to 0 \text{ as } \htt(M_{\boldsymbol{r}, \boldsymbol{s}, \lambda})\to\infty \text{ for all } n\geq 0\right\}.
$$
Its Fr\'echet topology is given by all the norms $\norm{\cdot}_n$.

\subsection{The quantum Arens-Michael envelope}\label{Arens-Michael} As an application of this PBW theorem we explain an analogy between our definition of $\Uqcap$ and the Arens-Michael envelope of the classical enveloping algebra $\wideparen{U(\g)}$, which is the completion of the enveloping algebra $U(\g)$ with respect to all the submultiplicative seminorms which extend the norm on $L$.

As a Fr\'echet space, $\wideparen{U_q}$ is the completion of $U_q$ with respect to the norms $\norm{\cdot}_n$ for $n\geq 0$, which are the norms on $U_q$ coming from the $\pi$-adic filtrations on the $U_n$. The completion of $U_q$ with respect to the single norm $\norm{\cdot}_n$ is then $\Uqnhat$. For example these norms take the following values:
$$
\norm{E_{\alpha}}_n=\norm{F_{\alpha}}_n=\abs{\pi}^{-n},\quad \norm{K_\lambda}_n=1\quad\text{for all simple root } \alpha\text{ and all } \lambda\in P.
$$
We now aim to show that $\wideparen{U_q}$ does not actually depend on the choice of such norms. To make this statement precise, we first consider the canonical norm $\norm{\cdot}$ on the Laurent polynomial ring $L[K_\lambda:\lambda\in P]$, namely the one obtained from giving the $\pi$-adic topology to $R[K_\lambda:\lambda\in P]$ and extending scalars. Hence we have $\norm{K_\lambda}=1$ for all $\lambda$ in $P$. Note that the norms $\norm{\cdot}_n$ are all extensions of $\norm{\cdot}$ to $U_q$.

We will now work in a more general context. Let $A$ be a $\pi$-torsion free $R$-algebra, and equip $A_L$ with the norm coming from the $\pi$-adic topology on $A$. Let $B$ be a $\pi$-torsion free $A$-algebra, and suppose that $B\cap A_L=A$, where we regard $A, A_L$ and $B$ as subalgebras of $B_L$. Recall that a seminorm $p$ on $B_L$ is called \emph{submultiplicative} if for all $x, y\in B_L$ we have $p(xy)\leq p(x)p(y)$ and $p(1)=1$.

\begin{prop}
For $A$ and $B$ as above, suppose that $B$ is generated as an $A$-algebra by a finite set of elements $x_1, \ldots, x_m\in B\setminus A$ which normalise $A$, i.e. $x_iA=Ax_i$ for all $i$. For each $1\leq i\leq m$, pick a positive integer $d_i$, and consider the $A$-filtration on $B$ given by assigning degree $d_i$ to $x_i$ for each $i$. Then, for this filtration, $\wideparen{B_L}$ is isomorphic to the completion of $B_L$ with respect to all submultiplicative seminorms which extend the norm on $A_L$.
\end{prop}

\begin{proof}
The filtration gives rise to a family of norms $\norm{\cdot}_n$ on $B_L$, which are just the extensions to $B_L$ of the norms coming from the $\pi$-adic topology on each of the deformations $B_n$. Since the $\pi$-adic filtration on $B_n$ is an algebra filtration, it follows that these norms are submultiplicative. Also, the $\pi$-adic topology on $B_n$ restricts to the $\pi$-adic topology on $A$ for all $n$ because $B\cap A_L=A$, and so these norms extend the norm on $A_L$. Hence, since $\wideparen{B_L}$ is the completion of $B_L$ with respect to the norms $\norm{\cdot}_n$, there is a canonical embedding $\mathfrak{B}\hookrightarrow \wideparen{B_L}$, where $\mathfrak{B}$ denotes the completion of $B_L$ with respect to all submultiplicative seminorms that extend the norm on $A_L$. Thus we just need to prove that this map is surjective.

Surjectivity will follow if we can show that given any submultiplicative seminorm $p$ on $B_L$ that extends the norm on $A_L$, there is some $n$ such that $p\leq \norm{\cdot}_n$. This in turn is equivalent to showing that the unit ball
$$
B(p;1)=\{x\in B_L: p(x)\leq 1\}
$$
contains the unit ball of $B_L$ with respect to $\norm{\cdot}_n$, i.e. contains $B_n$ for some $n$. Now note that since $p$ is submultiplicative and as it extends the norm on $A_L$, we have that $B(p;1)$ is an $A$-algebra. Moreover, by definition of $(F_\cdot)$, $B_n$ is the $A$-subalgebra of $B$ generated by the $\pi^{nd_i} x_i$. So we just need to show that there exists an $n\geq 0$ such that $\pi^{nd_i}x_i\in B(p;1)$ for all $i$. But that's clearly true since $p(\pi^{nd_i}x_i)=\abs{\pi}^{nd_i}p(x_i)\rightarrow 0$ as $n\rightarrow \infty$ for any $i$.
\end{proof}

\begin{cor}
$\wideparen{U_q}$ is isomorphic to the completion of $U_q$ with respect to all submultiplicative seminorms that extend $\norm{\cdot}$. Also, $\Oqcap$ is the completion of $\Oq$ with respect to all the submultiplicative seminorms that extend the norm on $L$.
\end{cor}

\begin{proof}
Set $A=R[K_\lambda:\lambda\in P]$ and $B=U_m$ for $\Uqcap$ (note that $B\cap A_L=A$ by Theorem \ref{PBW}), and $A=R$ and $B=\A$ for $\Oqcap$. The hypotheses of the Lemma are then satisfied.
\end{proof}

\subsection{Fr\'echet-Stein property of $\Uqcap$}\label{FreStProp} We can now start applying our techniques to $\Uqcap$.

\begin{lem}
The $R$-algebra $U_m$ satisfies conditions (i) and (ii) from \ref{Frechetcompl}.
\end{lem}

\begin{proof}
It is clear that $\gr_0 U_m$ is Noetherian. Moreover the height filtration on $U_m$ is the subspace filtration of the height filtration on $U_q$, thus there is a natural embedding $\gr U_m\hookrightarrow U^{(1)}$ where $U^{(1)}:=\gr U_q$. Write $U_m^{(1)}:=\gr U_m$. This shows that $U_m^{(1)}$ is $\pi$-torsion free, thus flat. Moreover since $U_m$ is free it is also $\pi$-adically separated. Therefore $U_m$ is a deformable $R$-algebra. Recall now that we defined in \ref{PrelimonUq} a $\Z^{2N}_{\geq 0}$-filtration on $U^{(1)}$. Using the above embedding, we may now give to $U_m^{(1)}$ the corresponding $\Z^{2N}_{\geq 0}$-filtration. We see from the relations in Theorem \ref{PrelimonUq} that the associated graded algebra of $U_m^{(1)}$ is then $q$-commutative, hence Noetherian by Lemma \ref{prelimonOq}. Therefore $U_m^{(1)}$ is Noetherian, and condition (i) is satisfied. Condition (ii) just follows from definition of the height filtration.
\end{proof}

\begin{remark} If we equip $U$ with the filtration from Remark \ref{proof_of_PBW}, it is then also true that it satisfies conditions (i) and (ii). Just as in our previous proofs, one also gets that this filtration is equal to the subspace filtration from the height filtration on $U_q$ and then the same proof as in the Lemma applies. However the Fr\'echet completion $\wideparen{U_L}$ that one gets that way is not the same as $\Uqcap$. Specifically, the norms defining $\wideparen{U_L}$ all have value 1 at the elements $[K_{\alpha_i};0]$, which is not true in $\Uqcap$. Now the triples $(E_{\alpha_i}, F_{\alpha_i}, [K_{\alpha_i};0])$ correspond under specialisation at 1 to the usual $\mathfrak{sl_2}$ triples $(e_i, f_i, h_i)$ (for the simple roots) in $\g$, and in the Arens-Michael envelope $\wideparen{U(\g)}$, the defining norms do not necessarily have value 1 at $h_i$. While we are not working with a truly generic quantum group, this analogy motivates our choice of working with $\Uqcap$. Note however that the theorem below is also true, with essentially the same proof, for $\wideparen{U_L}$.
\end{remark}

Before getting to the next result, we introduce some notation. Let $e_1, \ldots, e_n$ be the simple root vectors coming from the Serre presentation of $\g$, which can then be extended to a Chevalley basis $x_1,\ldots, x_N$ of $\n$. It follows from \cite[Theorem 25.2]{Humphreys1} that the $R$-span $\n_R$ of $x_1,\ldots, x_N$ is a Lie lattice in $\n$, i.e. a lattice that is also an $R$-Lie algebra, and we write $\n_k:=\n_R/\pi \n_R$, a nilpotent $k$-Lie algebra.

We let $U(\n_R)$ be the universal enveloping algebra of $\n_R$. For $n\geq 0$, we denote by $U(\n_R)_n$ the $R$-subalgebra of $U(\n_R)$ generated by all $\pi^n e_i$. It is the $n$-th deformation of $U(\n_R)$ with respect to the height filtration (which is not the same as the PBW filtration -- it is defined completely analogously as the height filtration on $U_q$). Moreover, $U(\n_R)_n$ is also the universal enveloping algebra of the $R$-Lie subalgebra of $\n_R$ generated by all $\pi^n e_i$. However, in light of the relations in \cite[Theorem 25.2]{Humphreys1}, we see that this $R$-Lie subalgebra is canonically isomorphic as an $R$-Lie algebra to $\n_R$ by mapping $\pi^ne_i\to e_i$, and hence there is a canonical isomorphism of $R$-algebras $U(\n_R)\cong U(\n_R)_n$ for all $n\geq 0$. Thus in particular we have that $U(\n_R)_n/\pi U(\n_R)_n\cong U(\n_k)$. In the light of these facts, we can now prove the following:

\begin{thm}
$\wideparen{U_q}$ is a Fr\'echet-Stein algebra.
\end{thm}

\begin{proof}
By Theorem \ref{defFreSt} and the previous Lemma, the result will follow if we prove that condition (iii) is satisfied in $U_n$ for all $n\geq m$. As before, we let $I=\pi U_n\cap U_{n+1}$. We know that $I$ is generated by $\pi$, $\pi^{(n+1)\htt\beta_i}E_{\beta_i}$ and $\pi^{(n+1)\htt\beta_j}F_{\beta_j}$ ($1\leq i, j\leq N$) by Proposition \ref{Frechetcompl}(ii). Observe that $\overline{\pi^{n+1}E_{\alpha_i}}$ commutes with $\overline{\pi^{n+1}F_{\alpha_j}}$ for all $i, j$ since $\pi^nE_{\alpha_i}$ and $\pi^nF_{\alpha_J}$ commute in $\gr U_n$, and so the same can be said of $\overline{\pi^{(m+1)\htt\beta_i}E_{\beta_i}}$ and $\overline{\pi^{(m+1)\htt\beta_j}F_{\beta_j}}$. Moreover we also have that all $\overline{\pi^{(m+1)\htt\beta_i}E_{\beta_i}}$ and $\overline{\pi^{(m+1)\htt\beta_j}F_{\beta_j}}$ $q$-commute with $\overline{K_\lambda}$ for all $\lambda\in P$.

Therefore it is enough to show that the elements $\overline{\pi^{(n+1)\htt\beta_i}E_{\beta_i}}$ for all $i$ form a polycentral sequence in $U_{n+1}^+/\pi U_{n+1}^+$, since the ideal $I$ is preserved by the automorphism $\omega$. But since $q\equiv 1 \pmod{\pi}$ we have a surjection
$$
U(\n_k)\cong U(\n_R)_{n+1}/\pi U(\n_R)_{n+1}\to U_{n+1}^+/\pi U_{n+1}^+
$$
from the universal enveloping algebra of $\n_k$, which sends $e_i$ to $\overline{\pi^{n+1}E_{\alpha_i}}$. In fact, by considering PBW bases we see that this is an isomorphism. Hence it suffices to show that the elements of the Chevalley basis in some order form a polycentral sequence in $U(\n_k)$. But that is a well known fact (and more generally any ideal of $U(\n_k)$ is polycentral by \cite[Theorem A]{Smith76}).
\end{proof}

By applying Corollary \ref{defFreSt} we immediately get:

\begin{cor} The natural map $U_q\to\Uqcap$ is flat.
\end{cor}

The Corollary gives an exact functor $M\mapsto \Uqcap\otimes_{U_q}M$ between the category of finitely generated $U_q$-modules and the category of coadmissible $\Uqcap$-modules. Indeed, since $U_q$ is Noetherian, a finitely generated module $M$ is finitely presented, hence so is $\Uqcap\otimes_{U_q}M$. In particular, since a finite dimensional $U_q$-module $V$ is complete with respect to any norm, we have $V\cong \Uqcap\otimes_{U_q} V$ and so $V$ is also coadmissible as a $\Uqcap$-module.

\subsection{Fr\'echet-Stein property of $\Oqcap$}\label{OqFS} As an $L$-algebra, $\Oq$ is generated by $x_1, \ldots, x_r$, i.e. by the matrix coefficients of the fundamental representations. Now the issue is that the $q$-commutator relations between these are not necessarily defined over $R$ here. Indeed recall from \ref{prelimonOq} that we have
$$
x_ix_j=q_{ij}x_jx_i+\sum_{s=1}^{j-1}\sum_{t=1}^{r} (\alpha^{st}_{ij} x_s x_t+ \beta^{st}_{ij}x_tx_s),
$$
for $1\leq j<i\leq r$ with $\alpha^{st}_{ij}, \beta^{st}_{ij}\in L$ for all $i,j,s,t$. These relations are obtained by considering $\mathcal{R}$-matrices for representations of $U_q$ and it is unclear to us whether the $\mathcal{R}$-matrices are the same when considering integral forms. Note however that the defining relations of $\Oq$ are defined over $R$ in type $A$ by \cite[Proposition 12.12]{Andersen}.

We fix this issue by deforming enough. Recall the filtration on $\Oq$ given by assigning to each $x_i$ degree $d_i=2^r-2^{r-i}$, where we had that whenever $i> j> s$ and $t\leq r$, we always have $d_i+d_j>d_s+d_t$. Thus we see that if we let $y_i=\pi^{ld_i}x_i$ for $l$ sufficiently large, multiplying the above relation by $\pi^{l(d_i+d_j)}$ yields
\begin{equation}\label{comrelinOq}
y_iy_j=q_{ij}y_jy_i+\sum_{s=1}^{j-1}\sum_{t=1}^{r} (\alpha'^{st}_{ij} y_s y_t+ \beta'^{st}_{ij}y_ty_s),
\end{equation}
where now $\alpha'^{st}_{ij}, \beta'^{st}_{ij}\in R$. Fix the smallest $l$ such that this holds and let $B$ be the $R$-subalgebra of $\Oq$ generated by $y_1, \ldots, y_r$.

\begin{lem}
The algebra $B$ is Noetherian, $\pi$-adically separated and $\pi$-torsion free.
\end{lem}

\begin{proof}
$B$ is $\pi$-torsion free because $\A$ is. Moreover, let $(F'_\cdot)$ be the filtration on $B$ given by assigning degree $d_i$ to each $ y_i$. Then with respect to that filtration, we see by the proof of \cite[Proposition I.8.17]{BroGoo02} that $\gr^{F'}B$ is $q$-commutative over $R$ and so is Noetherian by Lemma \ref{prelimonOq}. So we just need to show that it's $\pi$-adically separated. But that follows because $B\subseteq \A$ and $\A$ was $\pi$-adically separated.
\end{proof}

We now filter $B$ by assigning degree 1 to all the $y_i$'s. By Proposition \ref{Arens-Michael} we see that $\wideparen{\Oq}\cong\wideparen{B_L}$. Let $A=B_1$ be the first deformation of $B$, i.e. the $R$-subalgebra of $\Oq$ generated by $\pi y_1,\ldots, \pi y_r$. Completely analogously as in the Lemma, we see that $A$ is Noetherian, $\pi$-adically separated and $\pi$-torsion free. We now set a new filtration on $B$ by defining
$$
G_tB=A\cdot\{y_{i_1}a_{i_1}\cdots y_{i_l}a_{i_l}: a_{i_j}\in A \text{ and } \sum_{j=1}^l d_{i_j}\leq t\}.
$$
This is the smallest algebra filtration on $B$ such that $y_i\in G_{d_i}B$ and $A=G_0B$.

\begin{prop}
With respect to the above filtration, the associated graded ring $\gr^GB$ is finitely generated as an $A$-algebra by elements which $q$-commute with the $R$-algebra generators of $A$, and which also $q$-commute with each other.
\end{prop}

\begin{proof}
Set $z_i:=y_i+G_{d_i-1}B\in \gr^GB$ to be the symbol of $y_i$ for each $1\leq i\leq r$. Any homogeneous component $\gr^G_t B$, if it is non-zero, is spanned over $A$ by the symbols of the products $y_{i_1}a_{i_1}\cdots y_{i_l}a_{i_l}$ such that $\sum_{j=1}^l d_{i_j}=t$, and any such element equals $z_{i_1}a_{i_1}\cdots z_{i_l}a_{i_l}$. Therefore $\gr^GB$ is generated over $A$ by the $z_i$.

Now, for any $1\leq j<i\leq r$, we have
\begin{align*}
y_i(\pi y_j)-q_{ij}(\pi y_j)y_i&=(\pi y_i)y_j-q_{ij}y_j(\pi y_i)\\
&=\sum_{s=1}^{j-1}\sum_{t=1}^{r} \big(\alpha'^{st}_{ij} y_s (\pi y_t)+ \beta'^{st}_{ij}(\pi y_t)y_s\big)\in G_{d_j-1}B.
\end{align*}
Therefore we see that $z_i(\pi y_j)=q_{ij}(\pi y_j)z_i$ in $\gr^GB$ for all $i,j$, so that the $z_i$'s will $q$-commute with the generators of $A$. Furthermore we have $z_iz_j=q_{ij}z_jz_i$, i.e. the $z_i$'s will $q$-commute with each other in $\gr_G B$. Indeed this follows from (\ref{comrelinOq}) because the $d_i$'s were chosen so that whenever $i>j>s$ we have for any $1\leq t\leq r$ that $d_i+d_j>d_s+d_t$.
\end{proof}

\begin{thm} $\wideparen{\Oq}$ is a Fr\'echet-Stein algebra.
\end{thm}

\begin{proof}
By Proposition \ref{Emerton}, it follows from the previous Proposition that $\widehat{B_L}$ is right flat over $\widehat{A_L}$ and that they are both left Noetherian. Left flatness and right Noetherianity will follow by the same argument applied to $B^\text{op}$. Thus we see that $\widehat{B_L}$ is flat over $\widehat{A_L}$. For any $n\geq 1$, we can repeat the entire above arguments replacing $B$ by the $R$-algebra generated by $\pi^n y_i$ for all $i$, and $A$ by the $R$-algebra generated by $\pi^{n+1} y_i$ for all $i$.
\end{proof}

\section{Verma modules and category $\hat{\Of}$ for $\Uqcap$}

We now start discussing an analogue of category $\Of$ for $\Uqcap$. Most of the content of this Section is inspired by \cite{Schmidt2}, whose main theorem has a natural quantum analogue which we prove. In fact most of the arguments work identically to there, but we reproduce them for the convenience of the reader.

\subsection{Topologically semisimple $\widehat{U_q^0}$-modules}

We begin with a discussion of semisimplicity for modules over the algebra $\widehat{U_q^0}:=\widehat{U^0}\otimes_R L$. In our future paper \cite{Nico1} we will also need some of these results working with $\widehat{(U^{\text{res}}_R)^0_L}$ instead, where $(U^{\text{res}}_R)^0=U_q^0\cap U_R^{\text{res}}$. The proofs will be identical for either of them, so we will let $\Hh$ denote both of these to simplify notation. Our treatment is inspired by the work of F\'eaux de Lacroix \cite{Lacroix}.

First recall that given $\lambda\in P$, there is a character $\psi_\lambda$ of $U_q^0$ defined by $\psi_\lambda(K_\mu)=q^{\langle \lambda, \mu\rangle}$ for any $\mu\in P$, and the restriction of this character to $(U^{\text{res}}_R)^0$ has image in $R$ (see \cite[Lemma 1.1]{Andersen}). Given a $U_q^0$-module $M$, its \emph{$\lambda$-weight space} is defined to be
$$
M_\lambda=\{m\in M : um=\psi_\lambda (u)m\text{ for all } u\in U_q^0\}.
$$
Since $q$ is not a root of unity these are all linearly independent and the sum of the weight spaces in $M$ is direct.

We will now consider the category $\mathscr{M}(\Hh)$ whose objects are Fr\'echet spaces $\M$ endowed with an action of $\Hh$ by $L$-linear endomorphisms, and whose morphisms are continuous $L$-linear maps which preserve the action of $\Hh$. Given an object $\M$ of this category and $\lambda\in P$, we denote by $\M_\lambda$ the $\lambda$-weight space of $\M$ when viewed as a $U_q^0$-module.

\begin{definition}\label{semisimple1} We say that $\M$ as above is \emph{topologically $\Hh$-semisimple} if for every $m\in \M$ there exists a family $\{m_\lambda \in\M_\lambda\}_{\lambda\in P}$ such that $\sum_{\lambda\in P}m_\lambda$ converges to $m$ in $\M$.
\end{definition}

We want to investigate the full subcategory $\D(\Hh)$ of $\mathscr{M}(\Hh)$ whose objects are the topologically $\Hh$-semisimple modules. We first need a couple of preparatory results.

We identify the weight lattice $P$ with its image in the group of characters of $U_q^0$ via $\lambda\mapsto \psi_\lambda$. Let $x\in U_q^0$. For every $\lambda\in P$ we write $x(\lambda):=\psi_\lambda(x)\in L$. Note that if $x\in(U^{\text{res}}_R)^0$ or $U^0$, then $x(\lambda)\in R$ for all $\lambda\in P$. Let $q'=q^{1/d}$ so that $q^{\langle \lambda, \mu\rangle}\in (q')^{\Z}$ for any $\lambda, \mu\in P$.

\begin{lem}\label{integrable2} Let $r\in\N$, $m_1, \ldots, m_r\in\Z$ and $\omega_1, \ldots, \omega_r$ be (not necessarily distinct) fundamental weights. For each $\gamma\in P$, write $n_i(\gamma)=d\langle \gamma, \omega_i\rangle\in\Z$ and let
$$
P_\gamma(t)=\prod_{i=1}^{r} (t^{n_i(\gamma)}-(q')^{m_i})\in R[t, t^{-1}].
$$
Then, for every positive integer $a\geq 1$, the image of the set $\{P_\gamma(q'): \gamma\in P \}$ in $R/\pi^aR$ is finite.
\end{lem}

\begin{proof} First let $b=v_\pi(q'-1)>0$ and note that $b=v_\pi((q')^{-1}-1)$. Consider
$$
Q_\gamma(t)=\prod_{i=1}^{r} (t^{n_i(\gamma)-m_i}-1)\in R[t, t^{-1}].
$$
Then we see that $P_\gamma(q')=(q')^{m_1+\cdots+m_r}Q_\gamma(q')$, so that it suffices to show that the result holds for $Q_\gamma(t)$. Note that since $v_\pi((q')^m-1)\geq b\abs{m}$ for any $m\in\Z$, it follows that $Q_\gamma(q')\equiv 0\pmod{\pi^a}$ whenever $b\abs{n_i(\gamma)-m_i}\geq a$ for any $1\leq i\leq r$. Let
$$
X=\{(k_1,\ldots, k_r)\in\Z^r : b\abs{k_i}<a \text{ for all } 1\leq i\leq r\}
$$
and set
$$
M=\left\{\prod_{i=1}^{r} ((q')^{k_i}-1) : (k_1, \ldots, k_r)\in X \right\}\cup\{0\}.
$$
Then by the above observation we have that every $Q_\gamma(q')$ is congruent to an element of $M$ modulo $\pi^a$. The result follows since $M$ is finite.
\end{proof}

\begin{prop}\label{integrable3}
Suppose that $X$ is a finite subset of $P$ and let $\lambda\in P\setminus X$. Then there is an element $p\in U_q^0$ such that $p(P)\subset R$, $p(X)=0$ and $p(\lambda)=1$.
\end{prop}

\begin{proof}
For each $\mu\in X$, the character $\psi_\mu$ is determined by its action on the $K_{\varpi_i}$, so as $\lambda\neq\mu$ there must be some $h_\mu\in\{K_{\varpi_1},\ldots, K_{\varpi_n}\}$ such that $h_\mu(\lambda)\neq h_\mu(\mu)$. Consider the product
$$
x=\prod_{\mu\in X}(h_\mu-h_\mu(\mu))\in U^0.
$$
Note that $h_\mu(P)\subset R$ for every $\mu\in X$ and that, furthermore, the image of $h_\mu(P)$ in $k=R/\pi R$ is constant equal to 1 because $K_{\varpi_i}(\gamma)=q^{\langle\gamma,\varpi_i\rangle}\equiv 1 \pmod \pi$ for any $1\leq i\leq n$ and any $\gamma\in P$. So $x(X)=0$, $x(\lambda)\neq 0$ and $x(P)\subset R$, actually such that $x(P)$ has image zero in $k$. Hence there exists a maximal $N>0$ such that $y:=\pi^{-N}x$ still satisfies $y(P)\subset R$, and of course we still have $y(X)=0$ and $y(\lambda)\neq 0$.

Now note that if $y(\lambda)\in R^\times$, then $p=y(\lambda)^{-1}y$ satifies the required hypothesis. Otherwise, note that the set of residues of $y(P)$ in $R/\pi^aR$ is in bijection with the residues of $x(P)=\pi^N y(P)$ in $R/\pi^{N+a}R$, hence is finite for any $a\geq 1$ by the Lemma. Let $V$ be a finite set in $R$, containing 0, such that every element of $y(P)$ is congruent to a unique element of $Y$ modulo $\pi$, and set
$$
g=\pi^{-1} \prod_{v\in V} (t-v)\in L[t].
$$
Then $g(y(P))\subset R$, $g(y(X))=0$ and $v_\pi(g(y(\lambda)))=v_\pi(y(\lambda))-1$. Moreover the image of $g(y(P))$ in $R/\pi^aR$ is in bijection with the image of $\pi g(y(P))$ in $R/\pi^{a+1}R$, which is finite for every $a\geq 1$ since it was for $y(P)$. By induction, we can then find $h\in L[t]$ such that $p:=h(g(y))$ satisfies the required properties.
\end{proof}

\begin{thm}\label{integrable4}
Suppose that $\M\in\D(\Hh)$. Then for each $m\in \M $, there exists a unique family $(m_\lambda)_{\lambda\in P}$ with $m_\lambda\in\M_\lambda$ such that $\sum_{\lambda\in P}m_\lambda$ converges to $m$. Moreover, if $m\in \Nn$ where $\Nn$ is a closed $U_q^0$-invariant subspace, then each $m_\lambda\in \Nn$.
\end{thm}

\begin{proof}
We know by definition that there is a family $(m_\lambda)_{\lambda\in P}$ with $m_\lambda\in\M_\lambda$ such that $\sum_{\lambda\in P}m_\lambda$ converges to $m$. So we just need to prove uniqueness. Fix $\mu\in P$, and let $q_1\leq q_2\leq \cdots$ be a countable set of semi-norms defining the topology on $\M$, so that $\M\cong \varprojlim \M_{q_i}$.

Fix some $i\geq 1$. There is an ascending chain $S_1\subset S_2\subset\cdots$ of finite subsets of $P$ such that $\lambda\in P\setminus S_j$ implies that $q_i(m_\lambda)\leq 1/j$. By the Proposition, for every $j\geq 1$, there exists $p_j\in U_q^0$ such that $p_j(P)\subset R$, $p_j(S_j\setminus \{\mu\})=0$ and $p_j(\mu)=1$. Then we have
$$
p_j\cdot m=\sum_{\lambda\in P} p_j(\lambda)m_\lambda=m_\mu +\sum_{\lambda\in P\setminus S_j}p_j(\lambda)m_\lambda.
$$
By construction, $q_i(p_j(\lambda)m_\lambda)\leq q_i(m_\lambda)\leq 1/j$ for all $\lambda\in P\setminus S_j$. Hence $p_j\cdot m\to m_\mu$ in $\M_{q_i}$ as $j\to\infty$. So we see that the image of $m_\mu$ in $\M_{q_i}$ is uniquely determined by $m$ by uniqueness of limits. Since $i$ was arbitrary and since $\M\cong \varprojlim \M_{q_i}$, it follows that $m_\mu$ is uniquely determined by $m$.

For the last part, since $\Nn$ is closed and so complete, it follows that $\Nn_{q_i}$ is equal to the closure of $\Nn$ in $\M_{q_i}$ for each $i\geq 1$, and $\Nn\cong\varprojlim \Nn_{q_i}$. Now $\Nn$ is $U_q^0$-invariant, so for every $i\geq 1$ we have that the image of $m_\mu$ in $\M_{q_i}$ equals $\lim p_j\cdot m\in \Nn_{q_i}$. Hence $m_\mu\in\Nn$.
\end{proof}

\begin{remark} The ideas in the proofs of the Proposition and the Theorem were adapted for quantum groups from a proof that was communicated to us privately by Simon Wadsley.
\end{remark}

Given $\M\in\D(\Hh)$, we may form
$$
M^{\text{ss}}=\bigoplus_{\lambda\in P} M_\lambda
$$
which is a $U_q^0$-module. From the above, we immediately get the first part of the next result:

\begin{cor}\label{semisimple2} The category $\D(\Hh)$ is stable under passage to closed $\Hh$-submodules and to the corresponding quotients. Moreover, given $\M\in \D(\Hh)$ and a closed submodule $\Nn$, we have $(\M/\Nn)^{\text{\emph{ss}}}\cong \M^{\text{\emph{ss}}}/\Nn^{\text{\emph{ss}}}$.
\end{cor}

\begin{proof}
For the last part, for every $m\in\M$, write $\overline{m}$ for its image in the quotient $\M/\Nn$. Suppose that $\overline{m}\in (\M/\Nn)^{\text{ss}}$. By continuity of the quotient map, if $m=\sum_{\lambda\in P} m_\lambda$ converges then $\overline{m}=\sum_{\lambda\in P} \overline{m_\lambda}$ converges too, and that sum must be finite by the uniqueness of the decomposition from the Theorem. Thus there is a finite set $S\subset P$ such that, if $\lambda \in P\setminus S$, then $m_\lambda\in \Nn$. Hence if we write $m'=\sum_{\lambda\in S} m_\lambda\in \M^{\text{ss}}$, then $\overline{m'}=\overline{m}$. This shows that the map
$$
\M^{\text{ss}}\to (\M/\Nn)^{\text{ss}}
$$
is surjective. We now simply observe that its kernel is $\Nn^{\text{ss}}$.
\end{proof}

\subsection{A bijection between $U_q$-invariant subspaces} We need one other result to do with topologically semisimple modules. It is completely analogous to \cite[Satz 1.3.19 \& Kor. 1.3.22]{Lacroix}, but we give a proof nevertheless.

\begin{prop}\label{semisimple3} Suppose that $\M\in\D(\Hh)$. Then the assignement
$$
f:\Nn\mapsto \Nn\cap \M^{\text{ss}}
$$
defines an injective map between the set of closed $\Hh$-submodules of $\M$ and the set of abstract $U_q^0$-submodules of $\M^{\text{ss}}$, with left inverse given by passing to the closure in $\M$. Now assume furthermore that all the weight spaces $\M_\lambda$ are finite dimensional. Then $f$ is in fact surjective and so bijective. If additionally, $\M$ is also equipped with a $U_q$-action by continuous $L$-linear endomorphisms extending the $U_q^0$-action, then the bijection descends to a bijection between the $U_q$-invariant objects.
\end{prop}

\begin{proof}
For the first part, we must show that $\Nn=\overline{\Nn\cap \M^{\text{ss}}}$. Pick $m\in \Nn$. By Theorem \ref{integrable4}, we may write $m=\sum_{\lambda\in P}m_\lambda$ where $m_\lambda\in \Nn$ for each $\lambda\in P$. For each $n\in \N$, let
$$
P_n=\left\{\sum n_i\varpi_i\in P : \abs{n_i}\leq n\right\}.
$$
Since each $P_n$ is a finite set, we may define $m_n=\sum_{\lambda\in P_n} m_\lambda\in \Nn\cap \M^{\text{ss}}$. Then we have $m_n\to m$ as $n\to \infty$ and so $m\in \overline{\Nn\cap \M^{\text{ss}}}$. Thus we see that $\Nn\subseteq\overline{\Nn\cap \M^{\text{ss}}}$. The other inclusion is trivial.

Now assume all weight spaces are finite dimensional, and let $N\subseteq \M^{\text{ss}}$ be a $U_q^0$-submodule. Note that $N$ must be semisimple since $\M^{\text{ss}}$ is semisimple. The result will follow if we show that for such an $N$, we always have $N=\overline{N}\cap \M^{\text{ss}}$. To do that, we need to show that $\overline{N}\cap \M^{\text{ss}}$ is contained in $N$, the other inclusion being clear. So pick $m\in \overline{N}\cap \M^{\text{ss}}$. Then there is a sequence $(m_j)_{j\in\N}$ converging to $m$ such that $m_j\in N$ for all $j$. Since all the $m_j$ lie in $\M^{\text{ss}}$, we can find an ascending chain of finite subsets $S_j\subseteq P$ such that $m_j=\sum_{\lambda\in S_j}m_{\lambda, j}$ with $m_{\lambda, j}\in \M_{\lambda}$. We may also find a finite subset $S_0\subseteq P$ such that $m=\sum_{\lambda\in S_0}m_\lambda$ with $m_\lambda \in \M_\lambda$, and without loss of generality we may assume that $S_0\subseteq S_1$. Let $S=\bigcup_{j\geq 0}S_j$.

Now it follows from our assumption on weight spaces that any finite direct sum of weight spaces is finite dimensional, and hence the subspace topology on it is equivalent to the Banach space topology given by the max norm. In particular the projection map to any direct summand is continuous. Since $\M^{\text{ss}}$ is the direct limit of the these finite direct sums, we see that the projection map from $\M^{\text{ss}}$ to any direct summand is continuous, where $\M^{\text{ss}}$ is given the subspace topology. Hence we have that, for a fixed $\lambda\in S$, $m_{\lambda,j}$ converges to $m_\lambda$ (where $m_{\lambda,j}$, respectively $m_\lambda$, is understood to be zero when $\lambda\notin S_j$, respectively $\lambda\notin S_0$). But now $m_{\lambda,j}\in N\cap \M_\lambda$ for every $j$, and $N\cap \M_\lambda$ is finite dimensional hence complete. So we get that $m_\lambda\in N$ for every $\lambda\in S_0$ as required.

For the last part, we have that $\M^{\text{ss}}$ is then a $U_q$-submodule of $\M$, so that $\Nn\cap \M^{\text{ss}}$ is $U_q$-invariant whenever $\Nn$ is $U_q$-invariant. Also, $U_q$-invariant subspaces of $\M$ are preserved under passing to the closure. Hence the result follows immediately from the above.
\end{proof}

\subsection{Category $\hat{\Of}$}

We are now in a position where we can define an analogue of the BGG category $\Of$ for $\Uqcap$. First we recall that there is a category, that we denote by $\Of$, which is the full subcategory of the category of $U_q$-modules consisting of modules $M$ that satisfy the following:
\begin{itemize}
\item $M$ is finitely generated;
\item $M$ is the sum of its weight spaces, i.e. $M=\oplus_{\lambda\in P}M_\lambda$; and
\item $\dim_L U_q^+m<\infty$ for all $m\in M$.
\end{itemize}
This category is an analogue of the integral subcategory $\Of_{\text{int}}$ (i.e. the direct sum of all integral blocks) of the usual BGG category $\Of$ for the complex Lie algebra $\g$ (see \cite{catO}). Our category $\Of$ shares all the standard properties of $\Of_{\text{int}}$, see \cite[Section 6]{qcatO} and \cite[Chapters 9-10]{ChaPre}. In particular, all modules in $\Of$ have finite dimensional weight spaces and have finite length, the highest weight $U_q$-modules all belong to that category, are indecomposable and have a unique simple quotient, and $\Of$ splits into blocks
$$
\Of=\bigoplus_{\lambda\in -\rho+P^+} \Of^\lambda
$$
where $\rho$ is half the sum of the positive roots, and the block $\Of^\lambda$ consists of those modules from $\Of$ whose composition factors have highest weights in $W\cdot \lambda$.

Now we have for each $n\geq m$ that $U^0=R[K_\lambda: \lambda\in P]\subset U_n$ and from the PBW theorem (Theorem \ref{PBW}) we see that $\pi^aU_n\cap U^0=\pi^aU^0$ for every $a\geq 1$. Hence it follows that the subspace topology on $U^0$ of the $\pi$-adic topology on $U_n$ is the $\pi$-adic topology on $U^0$. Thus we see that the injection $U_q^0\subseteq U_q$ is strict (in fact an isometry) with respect to all the norms $\norm{\cdot}_n$ for $n\geq m$ on $U_q$ and the single gauge norm $\norm{\cdot}$ on $U_q^0$ associated to $U_q^0$. Hence there is a canonical strict embedding $\widehat{U_q^0}\hookrightarrow \Uqcap$.

Moreover, recall from the notion of a coadmissible module from Definition \ref{FS1} and the properties of the category $\C(\Uqcap)$ from Proposition \ref{FS2}. These modules have a Fr\'echet topology attached to them, making them by the above into $\widehat{U_q^0}$-modules where the action is by continuous $L$-linear endomorphisms. 

\begin{definition} The category $\hat{\Of}$ for $\Uqcap$ is defined to be the full subcategory of $\C(\Uqcap)$ consisting of coadmissible modules $\M$ satisfying:
\begin{enumerate}
\item $\M$ is topologically $\widehat{U_q^0}$-semisimple with weights contained in finitely many cosets of the form $\lambda-Q^+$, with $\lambda\in P$; and
\item all weight spaces of $\M$ are finite dimensional.
\end{enumerate}
\end{definition}

From Proposition \ref{FS2} and Corollary \ref{semisimple2}, we immediately get:

\begin{prop}\label{O1} Let $\M$ be an object of $\hat{\Of}$.
\begin{enumerate}
\item The direct sum of two objects in $\hat{\Of}$ is in $\hat{\Of}$;
\item the category $\hat{\Of}$ is an abelian subcategory of $\C(\Uqcap)$;
\item the sum of two coadmissible submodules of $\M$ is in $\hat{\Of}$;
\item any finitely generated submodule of $\M$ is in $\hat{\Of}$; and
\item Let $\Nn$ be a submodule of $\M$. Then the following are equivalent:
\begin{enumerate}
\item $\Nn$ is in $\hat{\Of}$;
\item $\M/\Nn$ is in $\hat{\Of}$; and
\item $\Nn$ is closed in the Fr\'echet topology of $\M$.
\end{enumerate}
\end{enumerate}
\end{prop}

We also record here the following fact:

\begin{lem}\label{bijection} Let $\M\in \hat{\Of}$. There is an inclusion preserving bijection between the subobjects of $\M$ in $\hat{\Of}$ and the $U_q$-submodules of $\M^{\text{\emph{ss}}}$.
\end{lem}

\begin{proof}
We see from Proposition \ref{semisimple3} that the map
$$
\Nn\mapsto \Nn\cap \M^{\text{ss}}
$$
gives an inclusion preserving bijection between the closed, $U_q$-invariant, $\widehat{U_q^0}$-submodules of $\M$ and the $U_q$-submodules of $\M^{\text{\emph{ss}}}$. But the former are just the closed $\Uqcap$-submodules of $\M$, which are just the subobjects in $\hat{\Of}$ by Proposition \ref{O1}(v).
\end{proof}

\subsection{Verma modules} We may now define the objects which play the role of Verma modules. For each $\lambda\in P$, there is a one dimensional $U_q^{\geq 0}$-module $L_\lambda$ given by $u\cdot 1=\psi_\lambda(u)$, where we extend $\psi_\lambda$ to a character of $U_q^{\geq 0}$ by setting it to be $0$ on $U_q^+$. We can then define a Verma module $M(\lambda):= U_q\otimes_{U_q^{\geq 0}}L_\lambda$.

We now let $I_\lambda$ be the left ideal of $\Uqcap$ generated by all $E_{\alpha_i}$, $K_{\varpi_i}-\lambda(K_{\varpi_i})$ ($1\leq i\leq n$). Since it is finitely generated, it must be a coadmissible module and hence the quotient $\Uqcap/I_\lambda$ is coadmissible as well.

\begin{definition} We define the \emph{Verma module with highest weight $\lambda$} for $\Uqcap$ to be the quotient $\wideparen{M(\lambda)}:=\Uqcap/I_\lambda$, which is a coadmissible module.
\end{definition}

Note that $\wideparen{M(\lambda)}\cong\Uqcap\otimes_{U_q}M(\lambda)$. Indeed, if $J_\lambda$ denotes the left ideal of $U_q$ generated by all $E_{\alpha_i}$, $K_{\varpi_i}-\lambda(K_{\varpi_i})$ ($1\leq i\leq n$), then we have a short exact sequence
$$
0\to J_\lambda\to U_q\to M(\lambda)\to 0
$$
of $U_q$-modules, and our claim follows by tensoring it with $\Uqcap$.

We now want to show that $\wideparen{M(\lambda)}$ is an object of our category. To do this, we will need a tensor product decomposition of $\Uqcap$. Consider the filtration on $U^-$ given by assigning each $F_{\alpha_i}$ degree 1 (this is the same as the height filtration by Corollary \ref{proof_of_PBW}). The $n$-th deformation of $U^-$ with respect to this filtration is just $U_n^-$ for each $n\geq 0$. For $n\geq m$, by the PBW theorem (Theorem \ref{PBW}), we have that $\pi^aU_n\cap U_n^-=\pi^aU_n^-$ for every $a\geq 0$, so that there is an isometric embedding
$$
\widehat{U_{q,n}^-}:=\widehat{U_n^-}\otimes_R L\hookrightarrow \Uqnhat.
$$
Hence if we let $\wideparen{U_q^-}:=\varprojlim \widehat{U_{q,n}^-}$, then there is a strict embedding $\wideparen{U_q^-}\hookrightarrow \Uqcap$. Using Corollary \ref{proof_of_PBW}, we may describe $\wideparen{U_q^-}$ explicitly as follows:
\begin{equation}\label{Uqminus}
\wideparen{U_q^-}=\left\{\sum_{\boldsymbol{r}} a_{\boldsymbol{r}}F_{\beta_1}^{r_1}\cdots F_{\beta_N}^{r_N}: \abs{\pi^{-n\htt(F^{\boldsymbol{r}})}a_{\boldsymbol{r},\boldsymbol{s},\lambda}}\to 0 \text{ as } \htt(F^{\boldsymbol{r}})\to\infty \text{ for all } n\geq 0\right\}.
\end{equation}
We may completely analogously define the positive subalgebra of $\Uqcap$.

We can also do a similar construction for the positive Borel. For each $n\geq m$, the inclusion $U_n^{\geq 0}\subseteq U_n$ induces an isometric embedding
$$
\widehat{U_{q,n}^{\geq 0}}:=\widehat{U_n^{\geq 0}}\otimes_R L\hookrightarrow \Uqnhat
$$
and passing to the inverse limit, this gives a strict embedding $\wideparen{U_q^{\geq 0}}\hookrightarrow\Uqcap$ where $\wideparen{U_q^{\geq 0}}=\varprojlim \widehat{U_{q,n}^{\geq 0}}$.

\begin{lem}\label{tensor} The multiplication map defines a topological isomorphism
$$
\ten{\wideparen{U_q^-}}{\wideparen{U_q^{\geq 0}}}\to \Uqcap
$$
of bimodules.
\end{lem}

\begin{proof}
The PBW theorem (Theorem \ref{PBW}) for $U_m$ gives an isomorphism
$$
U_m^-\otimes_R U_m^{\geq 0}\cong U_m
$$
of filtered $R$-modules. The result follows from Theorem \ref{Uqhopf}.
\end{proof}

Note that, for every $\lambda\in P$, the one-dimensional $U_q^{\geq 0}$-module $L_\lambda$ is complete with respect to any Hausdorff locally convex topology, and so naturally extends to a $\wideparen{U_q^{\geq 0}}$-module.

\begin{prop}\label{verma} The module $\wideparen{M(\lambda)}$ lies in $\hat{\Of}$ and $\wideparen{M(\lambda)}^{\text{\emph{ss}}}=M(\lambda)$. There is a canonical inclusion preserving bijection between the subobjects of $\wideparen{M(\lambda)}$ and the $U_q$-submodules of $M(\lambda)$. In particular, $\wideparen{M(\lambda)}$ is an irreducible object if and only if $M(\lambda)$ is irreducible as a $U_q$-module.
\end{prop}

\begin{proof}
From the definition, we see that $\wideparen{M(\lambda)}=\Uqcap\otimes_{\wideparen{U_q^{\geq 0}}} L_\lambda$, and its topology is the quotient topology coming from $\Uqcap$. Since it's therefore complete, it follows that $\wideparen{M(\lambda)}\cong\Uqcap\widehat{\otimes}_{\wideparen{U_q^{\geq 0}}} L_\lambda$. By the Lemma and using the fact that the projective tensor product is associative, we obtain an isomorphism
$$
\wideparen{M(\lambda)}\cong \ten{\wideparen{U_q^-}}{L_\lambda}\cong \wideparen{U_q^-}\otimes_L L_\lambda
$$
as left $\wideparen{U_q^-}$-modules. By considering now the $\widehat{U_q^0}$-action on this, and using the description of $\wideparen{U_q^-}$in (\ref{Uqminus}), we see that $\wideparen{M(\lambda)}\in\hat{\Of}$ and that $\wideparen{M(\lambda)}^{\text{ss}}=U_q^-\otimes_L L_\lambda=M(\lambda)$. The final two statements follow immediately from Lemma \ref{bijection}.
\end{proof}

\begin{cor} Let $\lambda\in P$. Then the following are equivalent:
\begin{itemize}
\item $\wideparen{M(\lambda)}$ is an irreducible object in $\hat{\Of}$.
\item For every positive root $\beta$, $\langle \lambda+\rho, \beta^\vee\rangle\notin \N$.
\end{itemize}
\end{cor}

\begin{proof}
This is just the condition for $M(\lambda)$ to be irreducible, see \cite[Corollary 10.1.11]{ChaPre}.
\end{proof}

\subsection{Highest weight modules} Having defined the Verma modules, we now look more generally at highest weight modules.

\begin{definition} Given a coadmissible $\Uqcap$-module $\M$ and $\lambda\in P$, an element $0\neq m\in\M_\lambda$ is called a \emph{maximal vector} of weight $\lambda$ if $U_q^+\cdot m=0$. We say $\M$ is a \emph{highest weight module with highest weight $\lambda$} if it is the cyclic $\Uqcap$-module on a maximal vector in $\M_\lambda$.
\end{definition}

The next result follows directly from the definition of $\wideparen{M(\lambda)}$:

\begin{lem} The coadmissible module $\wideparen{M(\lambda)}$ is a highest weight module with highest weight $\lambda$.
\end{lem}

Note more generally that it is immediate from the definition of $\hat{\Of}$ that every object of $\hat{\Of}$ contains a maximal vector. Hence by Proposition \ref{O1}(iv), every irreducible object in $\hat{\Of}$ is a highest weight module.

\begin{prop}\label{highestweight} Let $\M\in\C(\Uqcap)$ be a highest weight module on a maximal vector $m\in\M$ of weight $\lambda\in P$. We have the following:
\begin{enumerate}
\item $\M$ is topologically $\widehat{U_q^0}$-semisimple with weights contained in $\lambda-Q^+$.
\item The weight spaces of $\M$ are finite dimensional and $\dim_L \M_\lambda=1$. In particular, $\M\in\hat{\Of}$ and $\M$ has finite length in $\hat{\Of}$.
\item Each non-zero quotient of $\M$ by a coadmissible submodule is again a highest weight module.
\item Each coadmissible submodule of $\M$ generated by a maximal vector $m'\in\M_\mu$ for some $\mu<\lambda$ is proper. In particular, if $\M$ is an irreducible object in $\hat{\Of}$ then all its maximal vectors lie in $Lm$, and hence $\text{\emph{End}}_{\Uqcap}(\M)=L$.
\item $\M$ has a unique maximal subobject and a unique irreducible quotient object and, hence, is an indecomposable object.
\item Let $\Nn$ be another highest weight module of weight $\mu$. Then
$$
\text{\emph{dim}}_L\text{\emph{Hom}}_{\Uqcap}(\M, \Nn)<\infty.
$$
If $\lambda\neq \mu$ then $\M$ and $\Nn$ are not isomorphic. If $\M$ and $\Nn$ are simple objects and $\lambda=\mu$, then $\M\cong \Nn$.
\end{enumerate}
\end{prop}

\begin{proof}
By definition of highest weight modules, there is a surjection $\wideparen{M(\lambda)}\to \M$ which is a morphism in $\C(\Uqcap)$. Hence we see from Proposition \ref{O1}(v) that $\M\in\hat{\Of}$. From Corollary \ref{semisimple2} and Proposition \ref{verma}, we get a surjection
$$
M(\lambda)=\wideparen{M(\lambda)}^{\text{ss}}\to \M^{\text{ss}}.
$$
In particular, $\M^{\text{ss}}$ is a highest weight module of weight $\lambda$ in $\Of$. All properties therefore follow from the usual properties of $\Of$ by Lemma \ref{bijection}.
\end{proof}

If we write $\wideparen{V(\lambda)}$ to denote the unique irreducible quotient of $\wideparen{M(\lambda)}$, then we have $\wideparen{V(\lambda)}^{\text{ss}}\cong V(\lambda)$, where the latter denotes the unique irreducible quotient of $M(\lambda)$. Then we obtain:

\begin{cor} The map $\lambda\mapsto [\wideparen{V(\lambda)}]$ gives a bijection between $P$ and the set of isomorphism classes of irreducible objects in $\hat{\Of}$.
\end{cor}

\subsection{A functor $\Of\to\hat{\Of}$}

We now describe a functor between the categories $\Of$ and $\hat{\Of}$. It follows from Corollary \ref{flat} that the functor $M\mapsto \Uqcap\otimes_{U_q}M$ between the categories of $U_q$-modules and $\Uqcap$-modules is exact. If $M$ is a finitely generated $U_q$-modules, then $M$ is in fact finitely presented since $U_q$ is Noetherian and hence $\Uqcap\otimes_{U_q}M$ is also finitely presented. But this implies that $\Uqcap\otimes_{U_q}M$ is coadmissible. Thus there is an exact functor $F: M\mapsto \Uqcap\otimes_{U_q}M$ between the category of finitely generated $U_q$-modules and the category of coadmissible $\Uqcap$-modules.

Moreover we have already seen that $F(M(\lambda))=\wideparen{M(\lambda)}$. Thus, if $M\in\Of$ is a highest weight module of highest weight $\lambda$, then by exactness of $F$ we get that $F(M)$ is a quotient of $\wideparen{M(\lambda)}$ and hence is in $\hat{\Of}$. More generally, every object of $\Of$ has a finite filtration with highest weight subquotients. Hence there is a surjection $\oplus_i M_i\to M$ from a finite direct sum of highest weight modules to $M$, and since $F$ commutes with finite direct sums, it follows that $F(M)$ is a quotient of $\oplus_i F(M_i)$ and so lies in $\hat{\Of}$. Hence $F$ restricts to an exact functor
$$
F:\Of\to\hat{\Of}.
$$
Then we have:

\begin{prop}\label{Oandhat1} The functor $F:\Of\to\hat{\Of}$ is a fully faithful exact embedding with left inverse given by $\M\mapsto \M^{\text{\emph{ss}}}$.
\end{prop}

\begin{proof}
It suffices to show that there is an isomorphism $M\cong F(M)^{\text{ss}}$ natural in $M$. First observe that there is such a natural $U_q$-module map, given by $m\mapsto 1\otimes m$. If $M=M(\lambda)$ for some $\lambda\in P$, that map is an isomorphism by the proof of Proposition \ref{verma}. If $M$ is a highest weight module, we have a short exact sequence
$$
0\to N\to M(\lambda)\to M\to 0
$$
for some $\lambda\in P$. Writing $N$ as a subquotient of $U_q$ and using the fact that $\wideparen{M(\lambda)}$ is the completion of $U_q/J_\lambda$ with the quotient locally convex topology, we see that the image of the map $F(N)\to \wideparen{M(\lambda)}$ is the closure of $N$ in $\wideparen{M(\lambda)}$. Hence $N\cong F(N)^{\text{ss}}$ by Proposition \ref{semisimple3} and it follows that $M\cong F(M)^{\text{ss}}$ by exactness of the two functors. Now if $M$ is arbitrary, it has a filtration whose subquotients are highest weight modules. By induction we may assume $M$ is an extension of highest weight modules. Then the result follows by the Five Lemma.
\end{proof}

Moreover we can easily identify the essential image of the functor $F$:

\begin{lem}\label{Oandhat2} The essential image of $F$ is the full subcategory of $\hat{\Of}$ whose objects are those modules $\M\in\hat{\Of}$ which have a finite filtration
$$
0=\M_0\subset \M_1\subset \cdots\subset \M_r=\M
$$
by subobjects such that the quotient $\M_i/\M_{i-1}$ is a highest weight module for each $i\geq 1$.
\end{lem}

\begin{proof}
The essential image is contained in this since, for $M\in\Of$, we have an an analogous finite filtration in $\Of$ with subquotients equal to highest weight modules and so we obtain the filtration for $F(M)$ by applying $F$ to this filtration and using exactness. For the converse, suppose that $\M$ is as described. Then by exactness of $\M\mapsto \M^{\text{ss}}$ (Corollary \ref{semisimple2}) and by Proposition \ref{highestweight} and its proof, we see that $\M^{\text{ss}}\in \Of$. Thus it suffices to show that $F(\M^{\text{ss}})\cong \M$. Now by applying the functor $\Uqcap\otimes_{U_q}(\cdot)$ to the inclusion $\M^{\text{ss}}\subset \M$ and postcomposing with the action map $u\otimes m\mapsto um$, we get a morphism $F(\M^{\text{ss}})\to \M$ in $\hat{\Of}$. Let $\K$ and $\C$ denote its kernel and cokernel respectively. Then from Proposition \ref{Oandhat1} we get that $\K^{\text{ss}}=\C^{\text{ss}}=0$, and so $\K=\C=0$ by Proposition \ref{semisimple3}.
\end{proof}

We claim that the full subcategory described in Lemma \ref{Oandhat2} is the whole of $\hat{\Of}$:

\begin{thm}\label{mainthm} The functors $F$ and $(\cdot)^{\text{ss}}$ are quasi-inverse equivalence of categories between the categories $\Of$ and $\hat{\Of}$.
\end{thm}

The rest of this paper will be spent proving this theorem.

\subsection{Central characters} We now quickly recall some facts about central characters. Recall that the centre of $Z(U_q)$ is isomorphic to a polynomial algebra in $n$ variables (see \cite[Section 7.3, page 218]{josephbook} - note that this is only true for the simply connected form of the quantum group). For each $\lambda\in P$, $Z(U_q)$ acts on the Verma module $M(\lambda)$ by a central character $\chi_\lambda$ (see \cite[Lemma 6.3]{Jantzen}). These characters satisfy the usual property that $\chi_\lambda=\chi_\mu$ if and only if $\mu\in W\cdot\lambda$ (see \cite[Theorem 9.1.8]{ChaPre}) with respect the dot action $w\cdot\lambda=w(\lambda+\rho)-\rho$. Thus every character has a unique representative in $-\rho+P^+$.

For a given $\lambda\in -\rho+P^+$, the character $\chi_\lambda$ extends to a continuous character of the closure $\wideparen{Z(U_q)}$ of $Z(U_q)$ in $\Uqcap$, which we also denote by $\chi_\lambda$, using the fact that $\End_{\hat{\Of}}(\wideparen{M(\lambda)})=L$ from Proposition \ref{highestweight}(iv). Indeed it's clear from it that $\wideparen{Z(U_q)}$ acts on the Verma module by a continuous character, and we see that this character extends $\chi_\lambda$ by considering the semisimple part. Hence we see more generally from Proposition \ref{highestweight} that $\wideparen{Z(U_q)}$ acts on a highest weight module $\M$ by the character $\chi_\lambda$, and that every Jordan-Holder factor of $\M$ must necessarily have highest weight in $W\cdot \lambda$.

Now, if $\M\in\hat{\Of}$ then $Z(U_q)$ acts on each weight space $\M_\lambda$ and we may form the subspace
$$
\M_\lambda^{\chi}:=\{m\in \M_\lambda : (\ker \chi)^a\cdot m=0 \text{ for some }a=a(m)\geq 1\}
$$
where $\chi$ is a character of $Z(U_q)$. Since $\oplus_\lambda \M_\lambda^{\chi}$ is a $U_q$-submodule of $\M^{\text{ss}}$, its closure $\M^{\chi}$ inside $\M$ is a subobject in $\hat{\Of}$ by Lemma \ref{bijection}. Thus we may define the full subcategory $\hat{\Of}^\chi$ of $\hat{\Of}$ whose objects are those $\M\in\hat{\Of}$ such that $\M=\M^{\chi}$. When $\chi=\chi_\mu$ for some $\mu\in P$, we write $\hat{\Of}^\chi=\hat{\Of}^\mu$. We now establish a few facts about these subcategories.

\begin{lem}\label{chichimu} Suppose $\M\in\hat{\Of}$ and $\chi$ is a central character as above. If $\M^\chi\neq 0$, then $\chi=\chi_\mu$ for some $\mu\in P$.
\end{lem}

\begin{proof}
Since $\M^\chi$ is an object in $\hat{\Of}$, it must have a maximal vector $m\in \M_\mu^\chi$. Let $n\geq 1$ be minimal such that $(\ker \chi)^n\cdot m=0$. Pick $0\neq m'\in (\ker \chi)^{n-1}\cdot m$. Then $m'$ is still a maximal vector and the centre acts on it by $\chi$. On the other hand, the highest weight module generated by $m'$ is a quotient of $\wideparen{M(\mu)}$ and hence the centre acts on it by $\chi_\mu$. This forces $\chi=\chi_\mu$.
\end{proof}

Hence we see that the only such subcategories which are non-zero are the $\hat{\Of}^\mu$ for $\mu\in -\rho+P^+$.

\begin{prop}\label{artinian} For every $\mu\in -\rho+P^+$, the category $\hat{\Of}^\mu$ is abelian and the functor $\hat{\Of}\to \hat{\Of}^\mu$ given by $\M\mapsto \M^{\chi_\mu}$ is exact. Moreover, $\hat{\Of}^\mu$ is Artinian and Noetherian.
\end{prop}

\begin{proof}
Given a morphism $\M\to\Nn$ in $\hat{\Of}$ we have morphisms $\M_\lambda\to\Nn_\lambda$ for each $\lambda\in P$ and $\M_\lambda^{\chi_\mu}\to \Nn_\lambda^{\chi_\mu}$. Taking the sum over all $\lambda$ and passing to the closure, we see that the assignement $\M\mapsto \M^{\chi_\mu}$ is functorial. For the exactness, we apply the same argument again using the fact that module maps between coadmissible modules are automatically strict and so passage to the closure then preserves exactness by \cite[1.1.9, Corollary 6]{BGR}. As $\hat{\Of}^\mu$ is a full subcategory of $\hat{\Of}$, it is now clear that it is closed under passage to kernels and cokernels and, thus, abelian.

The last part follows using the classical argument for category $\Of$ (see \cite[Theorem 1.11]{catO}) as follows. Given $\M\in\hat{\Of}^\mu$, let $V=\sum_{\lambda\in W\cdot \mu} \M_\lambda$. Then $V$ is finite dimensional. Now if $0\neq \Nn'\subset \Nn$ is a strict inclusion of subobjects of $\M$, let $m\in \Nn_\lambda$ be such that its image in $\Nn/\Nn'$ is a maximal vector for some weight $\lambda$. The cyclic submodule of $\Nn/\Nn'$ generated by the image of $m$ is highest weight, hence $\wideparen{Z(U_q)}$ acts on it by $\chi_\lambda$. Hence it must be that $\chi_\lambda=\chi_\mu$ i.e. that $\lambda\in W\cdot \mu$. Thus by definition of $V$ we see that $m\in \Nn\cap V$ and so we obtain $\dim_L(\Nn\cap V)>\dim_L(\Nn'\cap V)$. The result now follows.
\end{proof}

They key step in the proof of Theorem \ref{mainthm} is the following:

\subsection{Proposition}\label{mainpt} \emph{The above functors $\hat{\Of}\to \hat{\Of}^\mu$ induce a faithful embedding of $\hat{\Of}$ into the direct product $\prod_{\mu \in -\rho+P^+} \hat{\Of}^\mu$.}

\begin{proof}
Choose polynomial generators $z_1, \ldots, z_n$ of $Z(U_q)$. Then for any $\M\in\hat{\Of}$, the vector space $\M_\lambda^{\chi_\mu}$ is the simultaneous generalised eigenspace of the finitely many commuting operators $z_1, \ldots, z_n$ with simultaneous generalised eigenvalues $\chi_\mu(z_1), \ldots, \chi_\mu(z_n)$. Now there is a finite field extension $L\subseteq L'$ such that
$$
\M_\lambda\otimes_L L'=\bigoplus_{\chi} (\M_\lambda\otimes_L L')^{\chi}
$$
where the sum runs over a finite number of $L'$-valued characters of $Z(U_q)$ and $(\M_\lambda\otimes_L L')^{\chi}$ is defined in the obvious way. Hence we just need to show that if $(\M_\lambda\otimes_L L')^{\chi}\neq 0$ then $\chi=\chi_\mu$ for some $\mu$. But this is Lemma \ref{chichimu}, noting that $\M\otimes_L L'$ is in $\hat{\Of}$ since $L'$ is a finite extension.

Thus we have that $\M_\lambda=\bigoplus_\mu \M_\lambda^{\chi_\mu}$. Moreover, the equality $\M^\mu \cap\M^{\text{ss}}=\bigoplus_\lambda \M_\lambda^{\chi_\mu}$ implies that $\M^{\text{ss}}=\bigoplus_\mu (\M^\mu\cap \M^{\text{ss}})$. Hence we see that from this and the usual properties of $(\cdot)^{\text{ss}}$ that the sum $\sum_\mu \M^\mu$ is direct and dense in $\M$. In particular the functor $\hat{\Of}\to\prod_\mu \hat{\Of}^\mu$ given by $\M\mapsto (\M^\mu)_\mu$ is faithful.
\end{proof}

We can now establish our main result. We first need a couple of preparatory results.

\begin{lem} For every $n\geq m$, there is a triangular decomposition
$$
\ten{\widehat{U_{q,n}^-}}{\ten{\widehat{U_q^0}}{\widehat{U_{q,n}^+}}}\overset{\cong}{\longrightarrow} \Uqnhat
$$
given by the multiplication map.
\end{lem}

\begin{proof}
By the PBW theorem (Theorem \ref{PBW}), the multiplication map yields a triangular decomposition
$$
U_n^-\otimes_R U^0\otimes_R U_n^+\overset{\cong}{\longrightarrow} U_n
$$
for every $n\geq m$. The result now follows by Proposition \ref{monoidal}.
\end{proof}

Given any coadmissible $\Uqcap$-module $\M$, we write $\M_n:=\Uqnhat\otimes_{\Uqcap}\M$ which is a finitely generated Banach $\Uqnhat$-module. Moreover the canonical map $\M\to \M_n$ has dense image. We also remark that the map $\Uqcap\to\Uqnhat$ is flat for every $n\geq m$ (see \cite[Remark 3.2]{SchTeit03}).

\subsection{Lemma} \emph{For any $\lambda\in P$ and any $n\geq m$, we have $\wideparen{V(\lambda)}_n\neq 0$.}

\begin{proof}
Consider the kernel $\K$ of the surjection $\wideparen{M_\lambda}\to\wideparen{V(\lambda)}$. Since $\Uqcap\to\Uqnhat$ is flat, the kernel of $(\wideparen{M_\lambda})_n\to\wideparen{V(\lambda)}_n$ is $\K_n$ for every $n\geq m$. By the triangular decomposition for $\Uqnhat$ from the previous Lemma, we get
$$
(\wideparen{M_\lambda})_n\cong\Uqnhat\otimes_{U_q} M_\lambda\cong \widehat{U_{q,n}^-}\otimes_L L_\lambda
$$
and so $(\wideparen{M_\lambda})_n$ is topologically $\widehat{U_q^0}$-semisimple with $((\wideparen{M_\lambda})_n)^{\text{ss}}=M_\lambda$. By Corollary \ref{semisimple2}, both $\K_n$ and $\wideparen{V(\lambda)}_n$ are topologically semisimple and it suffices to show that $\K_n^{\text{ss}}\neq ((\wideparen{M_\lambda})_n)^{\text{ss}}=M_\lambda$. Now the composite $\K^{\text{ss}}\subset \K\to \K_n$ has dense image, so it follows from Proposition \ref{semisimple3} that its image is $\K_n^{\text{ss}}$. So we get $\K_n^{\text{ss}}\cong \K^{\text{ss}}$ as $U_q^0$-modules, and now we see that $\K_n^{\text{ss}}\neq M_\lambda$ as required because $\wideparen{V(\lambda)}^{\text{ss}}\neq 0$.
\end{proof}

\begin{prop} The category $\hat{\Of}$ is Artinian and Noetherian.
\end{prop}

\begin{proof}
Let $\M\in\hat{\Of}$. We have from the proof of Proposition \ref{mainpt} that $\bigoplus_\mu \M^\mu$ is dense in $\M$. Now for any $n\geq m$, we have
$$
\M_n=\Uqnhat\otimes_{\Uqcap} \M\supseteq \Uqnhat\otimes_{\Uqcap} \left( \bigoplus _\mu \M^\mu\right)=\bigoplus_\mu (\M^\mu)_n.
$$
Any non-zero $\M^\mu$ has a composition series by Proposition \ref{artinian} and so $\wideparen{V(\lambda)}_n\subseteq (\M^\mu)_n$ for some $\lambda\in P$ and then we see that $(\M^\mu)_n\neq 0$ by the previous Lemma. Since $\M_n$ is a finitely generated $\Uqnhat$-module and $\Uqnhat$ is Noetherian, it follows that $\M^\mu=0$ for all but finitely many $\mu$. But then the sum $\bigoplus_\mu \M^\mu$ is finite and so closed by Proposition \ref{O1}(iii)\&(v).
\end{proof}

This now concludes the proof of Theorem \ref{mainthm}:

\begin{proof}[Proof of Theorem \ref{mainthm}]
This follows immediatly from the previous Proposition by Lemma \ref{Oandhat2}.
\end{proof}

\subsection{A Harish-Chandra isomorphism}

The analogue of Theorem \ref{mainthm} was proved for (non-quantum) Arens-Michael envelopes in \cite{Schmidt2}. One of the main ingredients was a version of the Harish-Chandra isomorphism. Recall that the centre of $Z(U_q)$ is isomorphic to a polynomial algebra in $n$ variables.

\begin{conj}\label{HarishChandra} The above isomorphism extends to a topological isomorphism $\wideparen{Z(U_q)}\to \Of(\mathbb{A}^{n,\text{\emph{an}}}_L)$ between the closure of $Z(U_q)$ in $\Uqcap$ and the algebra of rigid analytic functions on the analytification of affine $n$-space.
\end{conj}

To justify that this conjecture might plausibly be true, we show it for $U_q(\mathfrak{sl}_2)$. In that case, the centre $Z(U_q)$ is a polynomial algebra in the quantum Casimir element
$$
C_q:=FE+\frac{qK^2+q^{-1}K^{-2}}{(q-q^{-1})^2},
$$
see \cite[Proposition 2.18]{joseph}. In this $\mathfrak{sl}_2$ setting, recall that we had set the number $m$ to be the least positive integer such that
$$
\frac{\pi^{2m}}{q-q^{-1}}\in R.
$$
Having recalled this, we can now show:

\begin{prop}\label{sl2} Conjecture \ref{HarishChandra} holds for $U_q(\mathfrak{sl}_2)$.
\end{prop}

\begin{proof}
By definition of $C_q$, for $n\geq 2m$, we have
$$
\pi^{2n}C_q=(\pi^nF)(\pi^nE)+\frac{\pi^{2n}(qK^2+q^{-1}K^{-2})}{(q-q^{-1})^2}\in U_n.
$$
Hence we see that the subalgebra of $Z(U_q)$ consisting of polynomials in $\pi^{2n}C_q$ with coefficients in $R$ is contained in the centre of $U_n$. Conversely, suppose that $z=\sum_{i=0}^a c_i C_q^i\in U_n$. We show by induction on $a$ that each coefficient $c_i$ belongs to $\pi^{2ni}R$. If $a=0$ this is obvious so assume $a\geq 1$. Now note that
$$
C_q^a=F^aE^a+(\text{terms of lower height})
$$
with respect to the PBW basis for $U_q$. Indeed this follows from the commutator relation between $E$ and $F$. Thus we see that the coefficient of $F^aE^a$ in the basis expression for $z$ is $c_a$. But by the PBW theorem for $U_n$ (Theorem \ref{PBW}) it follows that the coefficient of $F^aE^a$ in the basis expression for $z$ is in $\pi^{2na}R$. Hence $c_a\in \pi^{2na}R$ and it follows that $c_aC_q^a\in R (\pi^{2n}C_q)^a\subseteq U_n$. Thus we may consider
$$
\sum_{i=0}^{a-1} c_i C_q^i=z-c_aC_q^a\in U_n
$$
and get that the other coefficients satisfy the required property by induction hypothesis.

The above calculation shows that the centre of $U_n$ is $Z_n:=R[\pi^{2n}C_q]$ for every $n\geq 2m$. If we write $\widehat{Z_{q,n}}:=\widehat{Z_n}\otimes_R L$, we get that the closure $\wideparen{Z(U_q)}$ of $Z(U_q)$ in $\Uqcap$ is the projective limit $\varprojlim \widehat{Z_{q,n}}$. From our description of $Z_n$, it is clear that this is isomorphic to $\Of(\mathbb{A}^{1,\text{an}}_L)$.
\end{proof}

\begin{remark}\label{HC2} The non-quantum version of Harish-Chandra for the Arens-Michael envelope is due to Kohlhaase \cite[Theorem 2.1.6]{kohlhaase}. A completely similar construction to the initial Harish-Chandra homomorphism applies to the Arens-Michael envelope, and he shows it to be an isomorphism. In our quantum setting, we can do that construction as well. One can straightforwardly construct a continuous projection map $\wideparen{Z(U_q)}\to\widehat{U_q^0}$ and twist by $-\rho$, which gives a continuous algebra homomorphism with image in the Weyl group invariants. However all the defining norms of $\Uqcap$ are identical on $\widehat{U_q^0}$ and so it is not clear a priori how to see the Fr\'echet structure of this image (this is something that does not occur in the classical situation).

The above calculation for $\mathfrak{sl}_2$ works because we have a complete and explicit description of the polynomial generator for the centre in terms of the PBW basis. In order to perform a similar calculation for a general Lie algebra, we'd need to have a similar description of the polynomial generators of the centre, something which we have not found in the literature.
\end{remark}

\bibliographystyle{abbrv}
\bibliography{bibforme}

\end{document}